\documentclass{amsart}

\usepackage[colorlinks,linkcolor=blue,anchorcolor=blue,citecolor=blue,breaklinks=true,unicode]{hyperref}

\usepackage{amscd}
\usepackage{amsmath, amssymb}
\usepackage{amsfonts}
\usepackage{color}

\renewcommand{\leq}{\leqslant}
\renewcommand{\geq}{\geqslant}
\renewcommand{\le}{\leqslant}
\renewcommand{\ge}{\geqslant}

\newcommand{\be}{\begin{equation}}
\newcommand{\ee}{\end{equation}}

\newcommand{\p}{\partial}
\newcommand{\ol}{\overline}
\newcommand{\ul}{\underline}

\begin{document}
\newtheorem{claim}{Claim}
\newtheorem{theorem}{Theorem}[section]
\newtheorem{lemma}[theorem]{Lemma}
\newtheorem{corollary}[theorem]{Corollary}
\newtheorem{proposition}[theorem]{Proposition}
\newtheorem{question}{question}[section]
\newtheorem{definition}[theorem]{Definition}
\newtheorem{remark}[theorem]{Remark}

\numberwithin{equation}{section}

\title[Prescribed curvature equations]
{The Dirichlet problem for hessian quotient type curvature equations in Minkowski space}

\author[M. Guo]{Mengru Guo}
\address{School of Mathematics, Harbin Institute of Technology,
	Harbin, Heilongjiang, 150001, China}
\email{22B912007@stu.hit.edu.cn}

\author[Y. Jiao]{Yang Jiao}
\address{School of Mathematical Sciences, Fudan University, Shanghai, 200433, China}
\email{jiaoy@fudan.edu.cn}
%\thanks{The second author is supported by the National Natural Science Foundation of China (Grant No. 12271126), the Natural Science Foundation of Heilongjiang Province (Grant No. YQ2022A006), and the Fundamental Research Funds for the Central Universities (Grant No. HIT.OCEF.2022030).}

\begin{abstract}
In this paper, we consider the Dirichlet problem for a class of prescribed Hessian quotient type curvature equations in Minkowski space. 
For non-convex domains, we prove the existence theorem by establishing the \emph{a priori} estimates without subsolution assumption and Serrin-type condition. 
%Due to the special structure of the equations, we drop the convexity restrictions on $\Omega$. 

\end{abstract}

\maketitle

\section{Introduction}

In Minkowski space $\mathbb{R}^{n,1}$, a function $u$ defined on a bounded domain $\Omega \subset \mathbb{R}^n$ is called spacelike if
\[
\sup_{\ol{\Omega}}|Du|<1.
\]
In this paper, we focus on the prescribed curvature problems for graph $M_u$ of a spacelike function $u$, where   
$$M_u:= \{(x, u(x)): x \in \Omega \subset \mathbb{R}^n \}.$$
We shall consider the Dirichlet problem for the following Hessian quotient type curvature equations ($k <n$):
\begin{equation}
	\label{GJ1.1}
	\frac{\sigma_k}{\sigma_l}(\lambda(\eta[M_u])) = \psi(X)>0\; \mbox{in} \;\Omega
\end{equation}
with $u=\varphi$ on $\partial \Omega$. Here $X=(x,u(x))$ is the position vector, $\varphi$ is a spacelike function which can be extended to $\overline{\Omega}$ in a suitable way, 
	$\sigma_{k}$ is the $k$-th elementary symmetric function, $\lambda(\eta[M_u])=(\lambda_1, \cdots, \lambda_n)$ denotes the eigenvalues of $\eta$ with respect to $g$ on $M_u$. 
 $\eta$ is the $(0,2)$-tensor field defined on $M$ by
\[
\eta = H g - h,
\] 
where $H (X)$ is the mean curvature of $M_u$ at $X \in M_u$, $g = \{g_{ij}\}$ is the induced Riemannian metric and $h = \{ h_{ij}\}$ is the second fundamental form.

When $k = n$ and $l = 0$, this problem was studied by H. Jiao  and the first author in \cite{GJ25} for convex and admissible domain.

In Euclidean space, Chen-Tu-Xiang \cite{CTX25} derived the existence and uniqueness results for the Dirichlet problem of 
\[
	\frac{\sigma_k}{\sigma_l}(\lambda(\eta[M_u])) = \psi(X, \nu)\geq 0,\;\; k<n
\]
with $\varphi=0$ in the degenerate case. 
The current work aims to extend the results of \cite{CTX25} from Euclidean space to Minkowski space  $\mathbb{R}^{n,1}$ with $\psi=\psi(X)>0$, whose induced metric is given by
\[
ds^2=dx^2_1+\cdots+dx^2_n-dx^2_{n+1}.
\]

We obtain three observations in this research. 
Firstly, when $k<n$, the convexity type assumption on the boundary is automatically satisfied.
As pointed out by Rirong Yuan, this is because the corresponding G\aa rding's cone is of Type 2 in the sense of \cite{CNSIII}. 
The relationship between \emph{the partial uniform ellipticity} of the equation and the corresponding cone is established in \cite{Yuan25b}(and references therein).
Besides, its connection to \emph{the tangent cone at infinity} (see \cite{GGQ22,GN23}) is discussed in \cite{Yuan25a}.
Secondly, the subsolution assumption can be weakened. Specifically, it suffices to assume that the spacelike boundary function \(\varphi\) can be extended to \(\overline{\Omega}\) in a suitable way, which is a necessary condition since the solution $u$ itself serves as such an extension.
This is only used to provide a starting point for the continuity method to obtain the existence result.
Thirdly, the special algebraic structure of \eqref{GJ1.1} for \(k<n\) simplifies the proof of boundary second-order double-normal estimates if \(\varphi=0\). Inspired by this idea, we further establish the corresponding estimates for general \(\varphi\) via Newton's formulas.

In contrast to the Euclidean case, 
when differentiating equations twice in the Minkowski space, uncontrollable terms including the square of curvatures will appear. This arises from the opposite sign of the Gauss equation.
Hence the Dirichlet problem for $k$-curvature equations 
\begin{equation}
	\label{GJ1.2}
	\sigma_{k} (\kappa[M_u]) = \psi
\end{equation}
remains unsolved for $3 \leq k \leq n - 3$, 
although the corresponding equations in the Euclidean context have long been well known, where 
\[
\sigma_{k} (\kappa[M_u]) = \sum_ {1 \leq i_{1} < \cdots < i_{k} \leq n}
\kappa_{i_{1}} \cdots \kappa_{i_{k}}, 
\]
$\kappa[M_u]=(\kappa_{1},\cdots,\kappa_{n})$ denotes the principal curvatures of $M_{u}$. 
Fortunately, based on the properties of \eqref{GJ1.1} in the corresponding G\aa rding's cone, we have no difficulty deriving global second-order estimates.
Therefore, we can solve the Dirichlet problem for such curvature equations in Minkowski space. 
The reader is referred to \cite{CNSV, CTX21, CTX20, GRW15, Ivochkina90, Ivochkina91, ILT96, LT94,   LRW17,  Zhou24} for the study of the Dirichlet problem for \eqref{GJ1.2} and the quotient type curvature equation in Euclidean space.

For the function $v(x) \in C^2(\Omega)$, we denote  $v_i = \frac{\p v}{\p x_i}$, $v_{ij} = \frac{\p^2 v}{\p x_i \p x_j}$, and use
$Dv = (v_{1}, \cdots, v_{n})$ and $D^2 v = \{v_{ij}\}$ to represent the ordinary gradient and Hessian matrix respectively.

Clearly, for the spacelike function $u$,  the Minkowski metric restricted to $M_u$ induces a Riemannian metric on $M_u$:
\[
g_{ij} = \delta_{ij} - u_i u_j, \ 1\leq i,j\leq n.
\]
The inverse of the metric is
\[
g^{ij} = \delta_{ij} + \frac{u_i u_j}{1 - |Du|^2}, \ 1\leq i,j\leq n.
\]
The second fundamental form of $M_u$ with respect to its upward unit normal vector field
\[\nu=\frac{(Du, 1)}{\sqrt{1-|Du|^2}}\]
is given by
\[h_{ij}=\frac{u_{ij}}{\sqrt{1-|Du|^2}}.\]

%Define the $(0, 2)$-tensor field $\eta$ on $M$ by\[\eta = H g - h,\] where $H (X)$ is the mean curvature of $M_u$ at $X = (x, u (x)) \in M_u$, $g = \{g_{ij}\}$, $h = \{ h_{ij}\}$. $\lambda(\eta[M_u]) = (\lambda_1, \cdots, \lambda_n)$ denote the eigenvalues of $\eta$ with respect to $g$ on $M_u$.  

In this paper, we consider the existence of smooth spacelike graphic hypersurfaces satisfying the Dirichlet problem
\begin{equation}
	\label{GJ1.3}
	\left\{ \begin{aligned}
		\frac{\sigma_k}{\sigma_l}(\lambda(\eta[M_u])) &= \psi(x, u) & \;\;~  \mbox{ in } \Omega, \\
		u &= \varphi& \;\;~  \mbox{ on } \partial \Omega,
	\end{aligned} \right.
\end{equation}
where $\psi \in C^\infty (\ol \Omega \times \mathbb{R}) > 0$, $\varphi\in C^{2}(\partial \Omega)$ is spacelike.

Define the G\aa rding's cone 
\[
\Gamma_k := \{\lambda = (\lambda_1, \cdots, \lambda_n) \in \mathbb R^n: \sigma_m(\lambda)>0,
m = 1, \cdots, k\}.
\]
We now introduce the definition of admissible set of equation \eqref{GJ1.1}.
\begin{definition}
	\label{admissible}

	A $C^2$ regular spacelike hypersurface $M$ is called $(\eta, k)$-convex if the principal curvatures
	$\kappa = (\kappa_1, \ldots, \kappa_n) \in \widetilde{\Gamma}_k$ at each $X \in M$, where $\widetilde{\Gamma}_k$ is the symmetric cone defined by
	\begin{equation}
		\label{GJ1.4}
		\widetilde{\Gamma}_k:= \{\kappa = (\kappa_1, \ldots, \kappa_n) \in \mathbb{R}^n: \lambda(\eta) = (\lambda_1, \cdots, \lambda_n) \in \Gamma_k, \lambda_i(\eta) = \sum_{j \neq i} \kappa_j \}.
	\end{equation}
	
	In addition, we call a $C^2$ spacelike function $u: \Omega \rightarrow \mathbb{R}$ admissible if its graph $M_u$ is $(\eta, k)$-convex, 
	and a function $\varphi \in C^2(\partial\Omega)$ has an admissible extension if its extension $\tilde{\varphi} \in C^2(\ol{\Omega})$ is admissible.
	%Such hypersurface was introduced by Sha \cite{Sha86} and Wu \cite{Wu87} to describe the boundaries of Riemannian manifolds which have the homotopy type of a CW-complex and was studied extensively in \cite{Sha87, HL13}.
\end{definition}

Jiao-Sun \cite{JaoSun22} considered the following degenerate curvature equation
\[
\det (\lambda( \eta[M_u])) = \psi(X, \nu)
\]
in Euclidean space with constant boundary condition. They derived the existence of $C^{1,1}$ regular graphical $(\eta, n)$-convex hypersurfaces for strictly convex $\Omega$. 
Motivated by \cite{JaoSun22}, Chen-Tu-Xiang \cite{CTX25} extended this work to a degenerate Hessian quotient type curvature equation \eqref{GJ1.1} with $0 \leq l < k < n$ and proved the existence and uniqueness of $C^{1,1}$ regular graphical hypersurfaces  in Euclidean space under the condition $\psi^{\frac{1}{k-1}}(x, u(x), \nu(x)) \in C^{1,1}(\Omega \times \mathbb{R} \times S^n)$. 
Building on \cite {CTX25}, 
we generalize their results to Minkowski space and establish the following theorem: 
\begin{theorem}
	\label{GJ-thm1}
	Let $k \geq 2,~ 0 \leq l < k <n$. Assume that $\Omega\subset\mathbb{R}^n$ is bounded domain with smooth boundary $\partial \Omega$. Suppose $\psi = \psi (x,z) \in C^\infty (\ol \Omega \times \mathbb{R}) > 0$, $\psi_z \geq 0$. %\textcolor{green}{and $\varphi$ is affine and spacelike.}  
	In addition, assume that there exists an admissible extension $\overline{\varphi} \in C^{2} (\overline{\Omega})$ of $\varphi$. 
	Then there exists a unique admissible solution $u \in C^{\infty} (\ol \Omega)$ to \eqref{GJ1.3}.
\end{theorem}

Generally speaking, people need some geometric conditions of $\Omega$ as
in \cite{Ivochkina91} to deal with equation \eqref{GJ1.1}. 
In Euclidean space, such convexity-type conditions are typically used to derive subsolutions or construct barriers in second order estimates on the boundary for Hessian equation and curvature equations in the form
$$
F(D^2 u)=f(\lambda(D^2 u))=\psi \;\mbox{and}\; F(\kappa(M))=\psi
$$
respectively. 
For equations like \eqref{GJ1.1}, the correspond condition is: 
there exists a
positive constant $K$ such that for each $x \in \partial\Omega$,
\begin{equation}
	\label{condition10}
	(\kappa_1'(x), \dots, \kappa_{n-1}'(x), K) \in \widetilde{\Gamma}_{k}, 
\end{equation}
where $\kappa_1'(x), \dots, \kappa_{n-1}'(x)$ are the principal curvatures of $\partial\Omega$ at $x$.
We note that for $k<n$, $\tilde{\Gamma}_k$ is the cone of type 2, so condition \eqref{condition10} holds automatically if K is sufficiently large.
%\begin{equation}	\label{condition11}	\sigma_{i}(\lambda(\kappa_{1}', \cdots \kappa_{n-1}', K))=C_{n-1}^{k} K^{k}+\mbox{other terms}(\sim K^{k-1}), \end{equation}if we let $K$ sufficiently large, we have where $\lambda=(\lambda_1,\cdots, \lambda_{n})$ is defined as in \eqref{GJ1.4}. It follows that, for any $K$ sufficiently large, $\sigma_{k}(\lambda(\kappa_{1}', \cdots \kappa_{n-1}', K))$ is always positive and thus, \[(\kappa_1'(x), \dots, \kappa_{n-1}'(x), K) \in \widetilde{\Gamma}_{q}, \;q=1,2,\ldots, n-1. \]In fact, since the positive $\kappa_i$ axes belong to $\tilde{\Gamma}_k$ for $k<n$, it is the cone of type 2 and condition \eqref{condition10} holds automatically.
Therefore, the geometric restrictions of $\Omega$ can be dropped in Theorem~\ref{GJ-thm1}.

However, in the Minkowski context, when we use condition \eqref{condition10} to construct subsolution as \cite{CNSIII}, it seems hard to control the aimed function to be spacelike. 
In fact, if we can construct subsolutions with any affine boundary value $\varphi$ for non-convex $\Omega$, 
then every subsolution is an admissible extension(spacelike, $\kappa(M_\varphi)\in \widetilde{\Gamma}_{k}$) of $\varphi$ on $\overline{\Omega}$.
By Lemma 2.1 in \cite{Bayard03}, $\Omega$ must be convex and this is contraction. 
This strong restriction on the domain does not arise in the Euclidean case.
Bartnik-Simon\cite{Bartnik82} used a barrier discussion when they study the prescribed mean curvature problem. 
By adapting their barriers, we derive the boundary gradient estimates without the subsolution assumption. 
Note that we need to assume $\varphi$ has an admissible extension to $\overline{\Omega}$ to provide a starting point for the continuity method used in the existence theorem, 
and this assumption is weaker but necessary compared with the subsolution condition.

Another unexpected thing of \eqref{GJ1.1} is its algebra structure when $k<n$. 
Generally speaking, we need the following  Serrin-type condition
\[
\frac{\sigma_k}{\sigma_l}(\kappa_1'(y), \dots, \kappa_{n-1}'(y), 0) \geq \psi(y,\varphi(y))\;, \forall y \in \partial\Omega
\]
to ensure the second order boundary estimates is reasonable. 
But for Dirichlet problem of  equation \eqref{GJ1.1},
when the boundary data $\varphi=0$, we found that \eqref{GJ1.1} have the following structure
\[
\psi=\Big(\frac{\sigma_{k}}{\sigma_{l} }\Big) (\lambda(\eta))
= \frac{C_{n-1}^k u_{nn}^k/(w^{3k}) +  O(|u_{nn}|^{k-1})}{C_{n-1}^l u_{nn}^l/(w^{3l}) + O(|u_{nn}|^{l-1})},\; k<n
\]
when $|u_{nn}| \rightarrow \infty$, if the double tangent and maxied tangent-normal estimates have been established. 
We guess that it has similar structure under general boundary data and proceed to derive
\[
	\psi=\Big(\frac{\sigma_{k}}{\sigma_{l} }\Big) (\lambda(\eta))
= \frac{C_{n-1}^k (w^2+|u_n|^2)^{k}u_{nn}^k/(w^{3k}) +  O(|u_{nn}|^{k-1})}{C_{n-1}^l (w^2+|u_n|^2)^{l} u_{nn}^l/(w^{3l}) + O(|u_{nn}|^{l-1})},\; k<n 
\]
with the help of Newton’s Formulas. 
The gradient estimates tells us $w$ has lower positive bound, 
thus the boundary estimates of second order are derived without Serrin-type condition. 
%\begin{theorem}	\label{GJ-thm2}	Let $k \geq 2,~ 0 \leq l < k <n$. Assume that bounded domain $\Omega\subset\mathbb{R}^n$ is $(\eta,k)$-convex with smooth boundary $\partial \Omega$. Suppose $\psi = \psi (x,z) \in C^\infty (\ol \Omega \times \mathbb{R}) > 0$, $\psi_z \geq 0$.  %and   $\varphi$ is affine. 	In addition, assume that there exists an admissible subsolution $\underline{u} \in C^{2} (\overline{\Omega})$ satisfying	$\kappa[M_{\underline{u}}]\in \widetilde{\Gamma}_k$ and	\begin{equation}		\label{sub}		\left\{		\begin{aligned}			&\frac{\sigma_k}{\sigma_l} (\lambda(\eta[M_{\underline{u}}])) \geq \psi(x, \underline{u}), &&in~\Omega,\\			&\underline{u}= 0, &&on~\partial \Omega.		\end{aligned}		\right.	\end{equation}	Then there exists a unique admissible solution $u \in C^{\infty} (\ol \Omega)$ to \eqref{GJ1.3}.\end{theorem}

To conclude the introduction, we briefly review some related studies.
If $\lambda(\eta [M_u])$ is replaced by the principal curvature $\kappa$ of $M_u$ and $l=0$, equation \eqref{GJ1.1} becomes the classical $k$-curvature equation \eqref{GJ1.2}. The mean curvature, scalar curvature and Gauss curvature equations correspond to $k = 1, 2,$ and $n$ in the formula \eqref{GJ1.2}, respectively.

The Dirichlet problem of the prescribed mean curvature equation in Minkowski space was first solved by Bartnik-Simon \cite{Bartnik82}, while Delanoè \cite{Delano90} studied the Dirichlet problem of the prescribed Gauss curvature equation (see also \cite{Guan98}). Bayard \cite{Bayard03} solved the Dirichlet problem for the prescribed scalar curvature equation in four dimension under the strict convexity condition, and Urbas \cite{Urbas03} extended this result to arbitrary dimensions.
Ren-Wang \cite{RW19, RW20} solved the case $k=n-1$ and $k=n-2$. For $3\leq k \leq n-3$, the problem remains open, due to the lack of \emph{a priori} $C^2$ estimates. 
%For this problem, Guan-Ren-Wang \cite{GRW15} given a second order estimates for admissible solutions of $k$-curvature equations \eqref{GJ1.2} whose graphs are $(k+1)$-convex. 
Wang \cite{WangZZ} obtained the curvature estimates for $(k+1)$-convex solutions of  $k$-curvature equations \eqref{GJ1.2} in Minkowski space,  
while Wang \cite{WangB25} established the estimates under semi-convex condition.

%In [Guan-Ren-Wang CPAM 2015] and [Li-Ren-Wang JFA 2016], they did the job in the Euclidean space. 

For the corresponding Hessian quotient type equations
 $$\frac{\sigma_{k}(\lambda (\Delta u I - D^2 u))}{\sigma_{l}(\lambda (\Delta u I - D^2 u))} = f(x, u(x), Du(x)),$$ 
 if $f(x, u(x), Du(x))=f(x)$ and $k<n$, it is included in \cite{CNSIII}. 
Liu-Mao-Zhao \cite{LMZ22} established Pogorelov type estimates of $k$-convex solutions with homogeneous boundary data in Lorentz-Minkowski space $\mathbb{R}^{n+1}$. 
Chen-Tu-Xiang \cite{CTX23.} considered the Dirichlet problem on compact Riemannian manifolds and obtained the existence result.

%In this paper, we establish a Pogorelov type estimate for second order derivatives (Theorem \ref{gj-thm1}) where we have used a method from\cite{Guan14} and \cite{Urbas02} to deal with bad third order terms. When the boundary data $\varphi$ is not constant, the estimates for double normal derivatives on the boundary become much more complicated. We shall use an idea of  \cite{Ivochkina91} and \cite{ILT96} to overcome these difficulties.

The rest of this paper is organized as follows. In Section 2, we provide some preliminaries. Sub- and super-barriers are constructed in Section 3. The 
$C^1$
estimate is established in Section 4. Boundary estimates for second-order derivatives are addressed for homogeneous and general boundary data in Sections 5 and 6, respectively. Finally, the proof of Theorem~\ref{GJ-thm1} is completed in Section 7.

\section{Preliminaries}

Let $\epsilon_1, \cdots, \epsilon_{n+1}$ be an orthonormal basis of $\mathbb{R}^{n+1}$. Consider a spacelike function $u$ and let $M_u$
represent its graph, which can be written as 
\[
M_u = \{X = (x, u(x)) | x \in \mathbb{R}^n\}.
\]
Then $u(x)= - \langle X, \epsilon_{n+1}\rangle$, where $\langle \cdot,\cdot \rangle$ is the Minkowski inner product.

On the other hand, when we choose local orthonormal frames $\{e_1,e_2,\cdots,e_n\}$ on $TM_u$. $\nabla$ denotes the induced Levi-Civita connection on $M$. For a smooth function $v$ on $M_u$, we denote $\nabla_i v=\nabla_{e_i}v,$
$\nabla_{ij} v = \nabla^2 v (e_i, e_j),$ etc in this paper.
Thus, we have
\[|\nabla u|=\sqrt{g^{ij}u_{i}u_{j}}=\frac{|Du|}{\sqrt{1-|Du|^2}}.\]
Under this normal coordinates $\{e_1,e_2,\cdots,e_n\}$,
we also have the following fundamental formulae and equations for the hypersurface $M$ in Minkowski space $\mathbb{R}^{n, 1}$:
\begin{equation}\label{Gauss}
	\begin{aligned}
		\nabla_{ij} X = \,& h_{ij}\nu \quad {\rm (Gauss\ formula)}\\
		\nabla_i \nu= \,& h_{ij} e_j \quad {\rm (Weigarten\ formula)}\\
		\nabla_k h_{ij} = \,& \nabla_j h_{ik} \quad {\rm (Codazzi\ equation)}\\
		R_{ijst} = \,& -(h_{is}h_{jt}-h_{it}h_{js})\quad {\rm (Gauss\ equation)},
	\end{aligned}
\end{equation}
where $h_{ij} = \langle D_{e_i}\nu,  e_j\rangle$ is the second fundamental form of $M$.

Define 
\begin{equation}
	\label{matrix}
      A = \{ a_{ij} \}=\Big\{  \frac{1}{w} \gamma^{ik} u_{kl} \gamma^{lj}  \Big\},
\end{equation}
where $w=\sqrt{1-|Du|^2}$ and $\gamma^{ik} =\delta_{ik}+\frac{u_iu_k}{w(1+w)}.$ 
Note that $\{\gamma^{ij}\}$ is invertible and its inverse is $\{\gamma_{ij}\}=\{\delta_{ij}-\frac{u_iu_j}{1+w}\}$, which is the square root of $\{g_{ij}\}.$

Clearly, $A = \{ a_{ij} \}$ is symmetric and the eigenvalues of $A$ equal the prescribed principal curvatures of $M_u$, which means
\begin{equation}
	\label{curvature}
	\kappa(\{a_{ij}\}) = \kappa(\{g^{ij}\} \{h_{ij}\}) = \kappa \big(\frac{1}{w} \left(I + \frac{Du \otimes Du}{w^2}\right) D^2 u \big).
\end{equation}

Note that the equation \eqref{GJ1.1} can be rewritten by 
\[
f(\kappa) := \frac{\sigma_k}{\sigma_l}(\lambda(\eta[M_u])) = \psi,
\]
where $\kappa = (\kappa_1, \ldots, \kappa_n)$ are the principal curvatures of $M_u$.
We need the following properties of $f (\kappa)$ (seeing \cite{JL20}). For any constant $A > 0$ and any compact set $K$ in $\widetilde{\Gamma}_k$
there is a number $R = R (A, K)$ such that
\begin{equation}
	\label{cj-2}
	f (\kappa_1, \ldots, \kappa_{n-1}, \kappa_n + R) \geq A, \mbox{ for all } \kappa \in K.
\end{equation}

In the final part of this section, we list some algebraic equalities and inequalities of $\sigma_k$. 
Denote $\sigma_{m;i_1,\cdots,i_k}(\kappa)=\sigma_m(\kappa)|_{\kappa_{i_{1}}=\cdots=\kappa_{i_{k}}=0}$ for integer
$1\leq i_1,\cdots,i_k\leq n, 1\leq m\leq n$
and
$n-k\leq m$.
Then we have, 
%More details can see \cite{CNS85, CDH23, Ger06, Hui99, Li96, LTR94, S05}.
\[
\frac{\partial\sigma_{k}}{\partial\kappa_{i}}(\kappa)=\sigma_{k-1;i}(\kappa),
\]
\[
\sum_i \sigma_{k-1; i}(\kappa) = (n-k+1) \sigma_{k-1}(\kappa)\]
and
\[
\sum_i \sigma_{k-1; i} (\kappa) \kappa_i = k \sigma_k (\kappa).
\]

The generalized Newton-MacLaurin inequality is as follows.
\begin{proposition}
	\label{NM.}
	For $\lambda \in \Gamma_k$ and $n \geq k > l \geq 0$, $n \geq r > s \geq 0$, $k \geq r$, $l \geq s$, we have
	\begin{align}
		\Bigg[\frac{{\sigma _k (\lambda )}/{C_n^k }}{{\sigma _l (\lambda
				)}/{C_n^l }}\Bigg]^{\frac{1}{k-l}} \le \Bigg[\frac{{\sigma _r
				(\lambda )}/{C_n^r }}{{\sigma _s (\lambda )}/{C_n^s
		}}\Bigg]^{\frac{1}{r-s}}. \notag
	\end{align}
\end{proposition}

\begin{proof}
	See \cite{S05}.
\end{proof}

Next, we need the following Newton's Formulas to connect a matrix $A$ and  $\sigma_{k}(\lambda(A))$, where $\lambda = (\lambda_1, \cdots,\lambda_n)$ is the eigenvalue vector of $A$.
% and the corresponding $\sigma_{k}$ whose variables are the eigenvalues of this matrix .
\begin{lemma}[Newton's Formulas]\label{nt}
	Define \( \sigma_k = 0 \) for \( k > n \).  
	Let 
	\[
	p_k = \lambda_1^k + \lambda_2^k + \cdots + \lambda_n^k, \qquad k \ge 0,
	\]
	be the sum of the \( k \)-th powers of $\{\lambda_{i}\}$.  
	
	Then the following identities hold for all \( k \ge 1 \):
	\[
	p_k - \sigma_1 p_{k-1} + \sigma_2 p_{k-2} - \cdots + (-1)^{k-1} \sigma_{k-1} p_1 + (-1)^k k \sigma_k = 0.
	\]
\end{lemma}

Finally, we list some important properties cone $\widetilde{\Gamma}_k$ defined by \eqref{GJ1.4}, more details can refer to \cite{CDH23}.
%As in \cite{CDH23}, we will introduce a new cone $\widetilde{\Gamma}_k$ defined by \eqref{wid} and list some important properties.
\begin{proposition}
	The following properties hold.
	
	(1) $\widetilde{\Gamma}_k$ are convex cones and
	\begin{equation}\label{px}
		\Gamma_1 = \widetilde{\Gamma}_1 \supset \widetilde{\Gamma}_2 \supset \cdots \supset \widetilde{\Gamma}_n \supset \Gamma_2 \supset \Gamma_3 \cdots \supset \Gamma_n.
	\end{equation}
	
%	(2) If $\kappa = (\kappa_1, \cdots, \kappa_n) \in \widetilde{\Gamma}_k$, then
%	$$
%	\frac{\partial \left[ \frac{\sigma_k (\eta)} {\sigma_l (\eta)} \right]} {\partial \kappa_i} \geq \frac{n (k - l)}{k (n-l)} \sum_{p \neq i} \frac{\sigma_{k-1; p} (\eta) \sigma_{l; p} (\eta)} {\sigma_l (\eta)^2},
%	$$
%	for any $i = 1, 2, \cdots, n$, where $0 \leq l < k \leq n$.
	
	(2) If $\kappa = (\kappa_1, \cdots, \kappa_n) \in \widetilde{\Gamma}_k$, then $\left[\frac{\sigma_k (\eta)}{\sigma_l (\eta)}\right]^{\frac{1}{k-l}} (0\leq l<k\leq n)$ are concave with respect to $\kappa$. Hence for any $(\xi_1,\cdots,\xi_n)$, we have
	\begin{equation}
		\label{pro3}
		\sum_{i, j} \frac{\partial^2 \left[\frac{\sigma_k (\eta)}{\sigma_l (\eta)}\right]} {\partial \kappa_i \partial  \kappa_j} \xi_i \xi_j 
		\leq 
		\left(1 - \frac{1}{k-l} \right) 
		\frac{\left[ \sum_i \frac{\partial \left(\frac{\sigma_k (\eta)}{\sigma_l (\eta)} \right)} {\partial \kappa_i}\xi_i\right]^2}{\frac{\sigma_k(\eta)}{\sigma_l(\eta)}}.
	\end{equation}
\end{proposition}

%For $\lambda=(\lambda_1,\cdots,\lambda_n)$, recall the notation $\eta=(\eta_1,\cdots,\eta_n)$ that $\eta_i=\sum_{j\neq i}\lambda_j$. Then we have the following basic properties. The proof can refer to Lemma 2.7 of \cite{CDH23}.
\begin{proposition}
	Suppose that $\kappa = (\kappa_1, \cdots, \kappa_n) \in \widetilde{\Gamma}_k$ are ordered with $\kappa_1 \geq \kappa_2 \geq \cdots \geq \kappa_n$, then
%	(1) $\lambda_1 \leq \lambda_2 \leq \cdots \leq \lambda_n$ and $\lambda_{n-k+1} > 0$.
%	(2) $\sigma_{k-1; n-k+1}(\eta) \geq c(n,k) \sigma_{k-1}(\eta), 	~\mbox{for}~0\leq l < k\leq n$.	(3) $\frac{\partial \left[\frac{\sigma_k (\eta)} {\sigma_l (\eta)} \right]} {\partial \lambda_1} \geq	\frac{\partial \left[\frac{\sigma_k (\eta)} {\sigma_l (\eta)} \right]} {\partial \lambda_2} \geq \cdots \geq	\frac{\partial \left[\frac{\sigma_k (\eta)} {\sigma_l (\eta)} \right]} {\partial \lambda_n},	~\mbox{for}~0\leq l<k\leq n$.	(4) $\frac{\partial \left[\frac{\sigma_k (\eta)} {\sigma_l (\eta)} \right]} {\partial \kappa_1} \leq	\frac{\partial \left[\frac{\sigma_k (\eta)} {\sigma_l (\eta)} \right]} {\partial \kappa_2} \leq \cdots \leq	\frac{\partial \left[\frac{\sigma_k (\eta)} {\sigma_l (\eta)} \right]} {\partial \kappa_n},	~\mbox{for}~0\leq l<k\leq n$.	
for $0 \leq l < k < n$, we have	
\begin{equation}		
	\label{ile}		
	\frac{\partial \left[\frac{\sigma_k (\eta)} {\sigma_l (\eta)} \right]} {\partial \kappa_i} 		\geq c(n,k,l) \sum_i \frac{\partial \left[\frac{\sigma_k (\eta)} {\sigma_l (\eta)} \right]} {\partial \kappa_i}		\geq c(n,k,l)\left[\frac{\sigma_k (\eta)}{\sigma_l (\eta)}\right]^ {1-\frac{1}{k-l}},		~\forall ~ 1 \leq i \leq n.
	\end{equation}
\end{proposition}

%Let $\mathcal{S}$ be the vector of $n\times n$ symmetric matrices and we extend the definition of $\Gamma_k$ into the matrices,
%\[\Gamma_k=\{A\in \mathcal{S}: \lambda(A)\in \Gamma_k\},\]
%where $\lambda(A)=(\lambda_1, \cdots, \lambda_n)$ denotes the eigenvalues of $A.$
%The equation \eqref{1-1} can be written as
%\begin{equation}\label{main equation graph}
%\sigma_k\Big(\frac{1}{w}\gamma^{ik}u_{kl}\gamma^{lj}\Big)=f(x,u).
%\end{equation}
%For $\kappa \in \Gamma$, let
%\begin{equation}
%\label{lambda}
%\lambda_i := \sum_{j \neq i} \kappa_j, i = 1, \ldots, n
%\end{equation}
%and
%\begin{equation}
%\label{def-h}
%f (\kappa) := \lambda_1 \cdots \lambda_n.
%\end{equation}
%Thus, we find
%\[
%K_\eta [M_u] = f (\kappa),
%\]

\section{Barrier constructions}

In this section, we shall construct barrier functions. 
Since $\Omega$ is a bounded domain, $u=\varphi$ on $\partial \Omega$ and $u$ is assumed to be spacelike, we have $|u|<C$ in $\ol\Omega$. 

First, we use the solvability of the prescribed mean curvature problem \cite{Bartnik82} in
order to construct an upper barrier at the boundary. Let $\overline{u}$ be the spacelike solution of the
Dirichlet problem
\[
\begin{cases}
	\sigma_1(\kappa[M_{\ol u}])) = \inf_{\Omega\times [\inf_{\ol{\Omega}}u, \sup_{\ol{\Omega}}u]} \psi>0 & \text{in } \Omega, \\[4pt]
	\overline{u} = \varphi & \text{on } \partial\Omega .
\end{cases}
\]
The comparison principle for the mean curvature operator gives $u \le \overline{u}$ in $\Omega$ and thus we get the upper barrier.

For the lower barrier, 
in \cite{Bartnik82}, Bartnik-Simon use the following barrier
\[
w^{\pm}(x)=w^{\pm}(\xi)\pm \int_{0}^{|x-\xi|} \frac{K-\frac{1}{n}\Lambda t^{n}}{\sqrt{t^{2n-2}+(K-\frac{1}{n}\Lambda t^{n})^{2}}} dt
\]
to derive the boundary gradient estimates in the prescribed mean curvature case. 
Here $K$ is a large constant to be determined, $\Lambda$ is the mean curvature of $w^{\pm}(x)$ and $|Dw^{\pm}|<1$. 
Note that $|Dw^{\pm}|$ could be close to $1$ by the choose of $K$, so it's a good choice of barrier. 
In fact, for a fixed point $x_{0}\in \partial \Omega$, choosing another point $\xi$ appropriately satisfying $|x_{0}-\xi|$ small, then
\[
v^{\pm}(x) := w^{\pm}(x)-w^{\pm}(x_{0})+\varphi(x_{0})
\]
are upper and lower barriers at $x_{0}$, they satisfy 
\[
	v^+(x) > \varphi(x), v^-(x)< \varphi(x), \forall x \in \partial\Omega\backslash \{x_{0}\}
\]
and 
\[
	v^{\pm}(x_0) = \varphi(x_0), \mathcal{M}(v^{\pm}) = \mp \Lambda.
\]
Since $\xi$ is out of $\overline{\Omega}$, $v^{\pm}(x)$ are smooth on $\overline{\Omega}$.
It follows from the standard comparison principle that $|Du|\leq 1-\delta<1$, where $\delta$ is relevant to $\varphi$.

We shall adapt their barriers and show they satisfy our requirement. For spherically symmetric function $w(r)$, consider the spherical coordinates with basis
\[
(\partial_{t},\partial_{r}, \partial_{\theta_{1}}, \ldots, \partial_{\theta_{n-1}})
\]
and
metric of $\mathbb{R}^{n,1}$
\[
ds^2=-dt^2+dr^2+r^2d\Omega, 
\]
where $d\Omega$ is the standard metric of $n-1$ dimension unit sphere. 
Write $d\Omega=s_{ij} d\theta_{i}d\theta_{j}(i,j=1,\ldots,n-1)$,  the non-zero Christoffel symbol are
\[
\Gamma_{ij}^{r}:=\Gamma_{\theta_{i}\theta_{j}}^{r}=-rs_{ij}. 
\]

The graph of $w$ is $(r,\theta_{1}, \ldots, \theta_{n-1},w(r))$. 
The corresponding tangent vectors are
\[
\tau_{r}=\partial_r+w'\partial_t, \tau_{\theta_i}=\partial_{\theta_i},\; i=1,\ldots,n-1.
\]
Consider the normalization tangent vectors and normal
\[
e_{r}=\frac{1}{\sqrt{1-(w')^2}}(1,0,\ldots,0,w'), e_{i}=\frac{1}{r\sqrt{s_{ii}}}\partial_{\theta_{i}}, i=1,2,\ldots, n-1, 
\]
\[
\vec{n}=\frac{1}{\sqrt{1-(w')^2}}(w',0,\ldots,0,1).
\]
Write $W=\sqrt{1-(w')^2}$, by direct calculation, we have 
\[
h_{ij}=(D_{e_{i}}e_{j},\vec{n})=\frac{w'}{rW}; h_{ri}=0;h_{rr}=\frac{w''}{W^{3}}. 
\]
Thus the principle curvatures are $(\frac{w'}{rW}, \cdots , \frac{w'}{rW}, \frac{w''}{W^3})$. 

Choosing 
\[
w(x)=w(\xi)- \int_{0}^{|x-\xi|} \frac{K-\frac{1}{m+1}t^{m+1}}{\sqrt{t^{2m}+(K-\frac{1}{m+1}t^{m+1})^2}} dt,
\]
then the principle curvatures of $M_{w(x)}$: $$\tilde{\kappa}:=(\tilde{\kappa}_1,\cdots,\tilde{\kappa}_n) = (-\frac{K}{r^{m+1}}+\frac{1}{m+1}, \cdots , -\frac{K}{r^{m+1}}+\frac{1}{m+1}, \frac{mK}{r^{m+1}}+\frac{1}{m+1}),$$
where $m, K$ are undetermined positive constants.
Note  that $\widetilde{\Gamma}_k$ is of type 2, fix $m$ large enough, we have 
\[
(-1,\cdots,-1,m)\in \widetilde{\Gamma}_k,
\]
which implies
\[
\tilde{\kappa}\in \widetilde{\Gamma}_k
\]
since $(\frac{1}{m+1}, \cdots , \frac{1}{m+1}, \frac{1}{m+1})\in \widetilde{\Gamma}_{k}$.
By \eqref{ile}, for the above $m$, choose $K\geq \sup_{\ol{\Omega}} r^{m+1}$ large enough, we have 
\[
\frac{\sigma_k}{\sigma_l}(\tilde{\eta}) > \sup_{\ol{\Omega}\times [ \inf_{\ol{\Omega}}u \sup_{\ol{\Omega}}u]} \psi,
\]
where $\tilde{\eta}_i = \sum_{j \neq i} \tilde{\kappa}_j$.

By the exact same discussion in the Appendix of \cite{Bartnik82}, we can show 
\begin{equation}
	\label{lower}
	\ul u := w(x)-w(x_{0})+\varphi(x_{0})
\end{equation}
is enough for a spacelike lower barrier at $x_{0}$.

\section{$C^1$ global estimates}

Next, we establish the $C^1$ estimates for the admissible solution to \eqref{GJ1.3}. Simliar as \cite{GJ25}, we prove the admissible solution $u$ is spacelike, which means there exists a constant $0 < \theta_0 < 1$ such that
\begin{equation}
	\label{spacelike}
	\sup_{\ol \Omega} |Du|\leq 1-\theta_0.
\end{equation}

This global gradient estimates will give an upper bound of $\tilde{w}$, which ensures the uniform ellipticity of $\frac{\sigma_{k}}{\sigma_{l}}(\lambda(\eta))$ if $||u||_{C^2(\overline{\Omega})}<C$.

%Since $M_u$ is $(\eta,k)$-convex, we have $\sigma_1(\lambda(\eta[M_u])) > 0$, which means $\sigma_{1}(\kappa(M_{u}))=H[M_u]>0$. 
%Consider the Dirichlet problem
%\begin{equation	\label{super}	\left\{ \begin{aligned}		H[M_{\bar{u}}] &= 0 & \;\;~  \mbox{ in } \Omega, \\		\bar{u} &= 0& \;\;~  \mbox{ on } \partial \Omega. 	\end{aligned} \right.\end{equation}
%By the solvability of the prescribed Lorentzian mean curvature problem \cite{Bartnik82} and the comparison principle for the mean curvature operator, 
%Obviously, $\bar{u} = 0$ is a supersolution of Dirichlet problem \eqref{GJ1.3}.
 %such that $u\leq \bar{u}$ in $\Omega$ and $u= \bar{u}$ on $\partial \Omega$. 
%Combining with $\psi_u\geq 0$, the subsolution condition \eqref{sub} and the maximum principle, we can find that

In section 3, we have constructed the upper and lower barrier functions, which means
\begin{equation}
	\label{jg-7}
	\sup_{\partial \Omega} |Du|\leq 1-\theta_0.
\end{equation}
%$$ \ul u \leq u \leq \ol u \text{ in } \Omega \text{ and } \ol u = u = \ul u \text{ on } \partial{\Omega},$$
%it follows that
%\[ \frac{\partial \ul u}{\partial \gamma} \leq \frac{\partial u}{\partial \gamma} \leq \frac{\partial \ol u}{\partial \gamma},\]
%where $\gamma$ is the interior unit normal to $\partial \Omega$.
%Then we have
%\begin{equation}
%	\label{jg-7'}
%	\sup_{\ol \Omega} |u| \leq \max\{\sup_{\ol \Omega} |\ol u|, \sup_{\ol \Omega} |\ul u|\} \leq C
%\end{equation}
%and
%\begin{equation}
%	\label{jg-7}
%	\sup_{\partial \Omega} |Du| \leq \max\{\sup_{\partial \Omega} |D \ul u|, \sup_{\partial \Omega} |D \ol u|\} \leq 1 - \theta
%\end{equation}
%for some constant $0 < \theta < 1$ since $\ul u$ and $\ol u$ are both spacelike.

%\section{$C^1$ estimates}

%Next, we establish the global $C^1$ estimates for the admissible solution to \eqref{GJ1.3}. Simliar as \cite{GJ25}, we prove the admissible solution $u$ is spacelike, which means there exists a constant $0 < \theta_0 < 1$ such that
%\begin{equation}
%	\label{spacelike}
%	\sup_{\ol \Omega} |Du|\leq 1-\theta_0.
%\end{equation}

Next, we prove an upper bound for  
$$ 
\tilde{w} =\frac{1}{\sqrt{1-|Du|^2}}=\frac{1}{w}.
$$

\begin{theorem}
\label{gradient}
Let $u \in C^3 (\Omega) \cap C^1 (\ol \Omega)$ be an admissible solution
of \eqref{GJ1.3}. Suppose the smooth function $\psi$ satisfies $\psi_u \geq 0$.
Then
\begin{equation}
\label{gradient-1}
\sup_{\ol \Omega} \tilde{w}
\leq C(1+\sup_{\partial \Omega} \tilde{w}),
\end{equation}
where $C$ depends on $|D\psi|$, the lower bound of $\psi$, $\sup_{\ol \Omega} |u|$ and other known data.
\end{theorem}
\begin{proof}
Take barrier function
\[
\hat{Q} := \tilde{w} e^{Bu},
\]
where $B$ is a positive constant to be determined later.
Suppose the maximum value of $\hat{Q}$ is achieved at an interior point $x_0 \in \Omega$. Then
\[
Q := \log \hat{Q} = \log \tilde{w} + Bu
\]
also attains its maximum at $x_0$. Let $\epsilon_1, \ldots, \epsilon_{n+1}$ be a standard basis of $\mathbb{R}^{n+1}$. We rotate the coordinate system $\epsilon_1, \ldots, \epsilon_{n}$ so that $u_1 (x_0) = |Du (x_0)|$ and $u_j (x_0) = 0$ for $j \geq 2$. 
Next, we consider on the manifold $M_u$. Define \[
e_i = \gamma^{is} \tilde{\partial}_s, \ i =1, \ldots,n,
\]
where $\gamma^{is} := \delta_{is}+\frac{u_iu_s}{w(1+w)}$ and $\tilde{\partial}_s := \epsilon_s + u_s \epsilon_{n+1}$. It's clear that 
$\{e_1,e_2,\cdots,e_n\}$ is an orthonormal frame on $TM_u$ around $X_0 = (x_0, u (x_0))$ such that, at $x_0$,
\begin{equation}
	\label{GJ3.1}
	\nabla_1u=\frac{|Du|}{w}=|\nabla u|, \nabla_i u=\frac{u_i}{w}=0, \text{ for } i\geq 2.
\end{equation}
We may further rotate $\epsilon_2, \ldots, \epsilon_n$ such that $\{u_{ij}\}_{i,j \geq 2}$ is
diagonal at $x_0$.
At $x_0$, we have
\begin{equation}
	0 = \nabla_i Q = \frac{\nabla _i \tilde{w}}{\tilde{w}} +B \nabla_i u\notag
\end{equation}
Note that by the Weingarten formula, 
$$
\nabla_i \tilde{w} = -\nabla_i \langle \nu, \epsilon_{n+1} \rangle = -\langle h_{ij} e_j, \epsilon_{n+1} \rangle
= - h_{ij} \langle \gamma^{js} \tilde {\partial_s}, \epsilon_{n+1}\rangle
$$
and 
\[
\langle \gamma^{js} \tilde {\partial_s}, \epsilon_{n+1}\rangle = - \gamma^{js} u_s = - \frac{u_j}{w} = - \nabla_j u.
\]
Combing \eqref{GJ3.1}, we have, at $x_0$
\begin{equation}
	\label{GJ3.2}
	0 = \nabla_i Q = \frac{h_{i1} \nabla_1 u}{\tilde{w}} + B \nabla_i u
\end{equation}
and 
\begin{equation}
	\label{GJ3.3}
	0 \geq F^{ij} \nabla_{ij} Q =
	\frac{F^{ij} \nabla_j h_{i1} \nabla_1 u + F^{ij} h_{il} \nabla_{jl} u}{\tilde{w}}
	- \frac{F^{ij} (h_{i1} \nabla_1 u) (h_{j1} \nabla_1 u) }{\tilde{w}^2} + B F^{ij}\nabla_{ij} u,
\end{equation}
where
\[
F^{ij} := \frac{\frac{\sigma_{k}}{\sigma_{l}}(\lambda(\eta))}{\partial h_{ij}}.
\]
In the rest of the proof, we declare that all computations are carried out at $X_0$.
From \eqref{GJ3.2}, we have
\[
h_{11} = -B \tilde{w} \mbox{ and } h_{1i} = 0 \mbox{ for } i \geq 2.
\]
Since at $X_0$,
\[
h_{11} = \frac{u_{11}}{w^3}, h_{1i} = \frac{u_{1i}}{w^2} \mbox{ and }
h_{ij} = \frac{u_{ij}}{w}
\mbox{ for } i, j \geq 2,
\]
we find that the matrix $\{h_{ij}\}$ is diagonal at $X_0$, and so is $\{F^{ij}\}$.
By the Codazzi equation and differentiating the equation \eqref{GJ1.1}, we have
\begin{equation}
	\label{GJ3.5}
	\begin{aligned}
		F^{ii} \nabla_i h_{i1} = \,& F^{ii} \nabla_1 h_{ii} = \nabla_1 \psi
		= \psi_{x_j} \nabla_1 x_j + \psi_u \nabla_1 u \\
		= \,& \frac{\psi_{x_1}}{w} + \psi_u \nabla_1 u
		= \tilde{w} (\psi_{x_1} +\psi_u u_1 )
		\geq \tilde{w} {\psi_{x_1}},
	\end{aligned}
\end{equation}
where the last inequality is due to that $\psi_u \geq 0$.
By the Gauss formula, we have
\begin{equation}
	\label{GJ3.6}
	\nabla_{ij} u=-\nabla_{ij}\langle X,\epsilon_{n+1}\rangle=-\langle \nabla_{ij}X,\epsilon_{n+1}\rangle =
	-h_{ij}\langle \nu,\epsilon_{n+1}\rangle = \tilde{w} h_{ij}.
\end{equation}
Combining \eqref{GJ3.1}-\eqref{GJ3.6} and \eqref{ile}, we obtain
\begin{equation}
	\label{GJ3.7}
	\begin{aligned}
		0 \geq \,& \psi_{x_1} u_1 \tilde{w}
		+ F^{ii} h_{ii}^2
		- B^2 F^{11} (\nabla_1 u)^2 + B (k-l) \psi \tilde{w}\\
		\geq \,& - |D \psi|\cdot |Du| \tilde{w} + B^2 \tilde{w}^2 F^{11} - B^2 F^{11} (\nabla_1 u)^2 + B (k-l) \psi \tilde{w}\\
		\geq \,& - |D \psi| \tilde{w} + B^2 F^{11} + B (k-l) \psi \tilde{w}\\ 
		\geq \,& (B (k-l) \psi - |D \psi|) \tilde{w}.
	\end{aligned}
\end{equation}
Then we get a contradiction provided
\[
B > \frac{\sup_{\ol \Omega}|D\psi|}{(k-l) \inf_{\ol{\Omega}}\psi}
\]
and
\eqref{gradient-1} follows as \cite{Bayard03}. It follows from \eqref{jg-7} and \eqref{gradient-1} that \eqref{spacelike} holds.
\end{proof}

%\begin{remark}	\label{remark}\textcolor{red}{Although Chen-Tu-Xiang \cite{CTX25} has given the gradient estimate for admissible solutions of the equation }	\[	\frac{\sigma_k}{\sigma_l}(\lambda(\eta[M_u])) = \psi(x, u, \nu)\geq 0	\]in Euclidean space, gradient estimates in Minkowski space remain a challenge for degenerate problems due to the metric. In addition, no satisfactory approach has been developed to control the term $\nu$, even in non-degenerate cases.\end{remark}

\section{$C^2$ estimates for homogeneous boundary data}

We shall first estimate the second derivatives of $u$ on the boundary.

\begin{theorem}
\label{thm-boundary}
Suppose  $\Omega\subset\mathbb{R}^n$ is a bounded domain with smooth boundary $\partial \Omega$. 
Let $u \in C^3 (\overline{\Omega})$ be an admissible solution of \eqref{GJ1.3}. Then there
exists a positive constant $C$ depending only on $\theta_0$,
$\|\psi\|_{C^{1} (\overline{\Omega} \times [-\mu_0, \mu_0]) }$, $\|u\|_{C^1(\overline \Omega)}$ and $\partial \Omega$
satisfying
\begin{equation}
\label{B2-0}
\max_{\partial \Omega} |D^2 u| \leq C,
\end{equation}
where $\mu_0 := \|u\|_{C^0 (\overline{\Omega})}$.
\end{theorem}

We will divide this estimate \eqref{B2-0} into three parts for proof: the double tangential derivative estimate, the mixed tangential-normal derivative estimate, and the double normal derivative estimate on $\partial \Omega$.

For any point $x_0 \in \partial \Omega$, without loss of generality, we may assume that $x_0$ is the origin and that the
positive $x_n$-axis is the inner normal direction to $\partial \Omega$ at the origin.
Furthermore, we may suppose that in a neighbourhood of the origin, the boundary $\partial \Omega$ is given by
\begin{equation}
\label{BC2-1}
x_n = \rho (x') = \frac{1}{2} \sum_{\alpha < n} \kappa'_\alpha x_\alpha^2 + O (|x'|^3),
\end{equation}
where $\kappa'_1, \ldots, \kappa'_{n-1}$ are the principal curvatures of $\partial \Omega$ at the origin as before and $x' = (x_1, \ldots, x_{n-1})$.
Define
$$\omega_\delta = \{x \in \Omega: \rho (x') < x_n < \rho (x') + \delta^2 , |x'| < \delta\},$$
we can find that the boundary $\partial \omega_\delta$ consists of three parts:
$$\partial \omega_\delta
= \partial_1 \omega_\delta \cup \partial_2 \omega_\delta \cup \partial_3 \omega_\delta,$$ where
$\partial_1 \omega_\delta$, $\partial_2 \omega_\delta$ and $\partial_3 \omega_\delta$  are defined by $\{x_n=\rho\} \cap\overline{\omega}_{\delta}$, $\{ x_n=\rho+\delta^2\}\cap\overline{\omega}_{\delta}$
and $\{|x'| = \delta\}\cap\overline{\omega}_{\delta}$ respectively.

\noindent
{\bf Part 1.} The double tangential derivative estimate on $\partial \Omega$, that is \begin{equation}
	\label{BC2-2}
	|u_{\alpha \beta} (0)| \leq C \mbox{  for } 1\leq \alpha, \beta \leq n - 1.
\end{equation}

Since \eqref{BC2-1}, We can obtain \eqref{BC2-2} by differentiating the boundary condition $u = 0$ twice on $\partial \Omega$.

\noindent
{\bf Part 2.} The mixed tangential-normal derivative estimate on $\partial \Omega$, that is 
\begin{equation}
	\label{BC2-3}
	|u_{\alpha n} (0)| \leq C \mbox{  for } 1\leq \alpha \leq n - 1.
\end{equation}

We rewrite the equation \eqref{GJ1.1} by the form
\begin{equation}
\label{1-1-1}
G (D^2 u, Du) := \frac{\sigma_{k}}{\sigma_{l}}(\lambda (\eta (M_u))) = f (\lambda (A[u])) = \psi (x, u),
\end{equation}
where $G = G (r, p)$ is viewed as a function of $(r, p)$ for $r \in S^{n \times n}$ and $p \in \mathbb{R}^n$.
Define
\begin{equation}
\label{BC2-25}
G^{ij} = \frac{\partial G}{\partial r_{ij}} (D^2 u, D u),\ \ G^{i} = \frac{\partial G}{\partial p_i} (D^2 u, D u)
\end{equation}
and the linearized operator by
\[
L = G^{ij} \partial_{ij}.
\]
We need the following lemma:
%Similar to lemma 2.3 of \cite{GS04}, we have the following lemma.
\begin{lemma} We have
\label{lemGS2}
\begin{equation}
\label{GS-2}
G^s = \frac{u_s}{w^2} \sum_i f_i \kappa_i + \frac{2}{w (1+w)} \sum_{t,j}F^{ij} a_{it} \big(w u_t \gamma^{sj}
   + u_j \gamma^{ts}\big),
\end{equation}
where $w = \sqrt{1 - |Du|^2}$, $a_{ij} =\frac{1}{w}\gamma^{ik}u_{kl}\gamma^{lj}$, $\kappa = \lambda (\{a_{ij}\})$, $f_i = \frac{\partial f (\kappa)}{\kappa_i}$ and
\[
F^{ij} = \frac{\partial f (\lambda (A[u]))}{\partial a_{ij}}.
\]
\end{lemma}
The details of the proof can be found in \cite{GJ25}.

Besides, to prove \eqref{BC2-3}, we shall use the strategy of \cite{Ivochkina90} to consider the function
\[
W := \nabla'_\alpha u  - \frac{1}{2} \sum_{1\leq \beta \leq n - 1}u_\beta^2
\]
defined on $\ol \omega_\delta$ for small $\delta$,
where
\[
\nabla'_\alpha u := u_\alpha + \rho_\alpha u_n, \mbox{ for } 1\leq \alpha \leq n - 1.
\]
Then we have
\[
W_s = u_{\alpha s} + \rho_{\alpha s} u_n  + \rho_{\alpha} u_{ns} - \sum_{\beta=1}^{n-1} u_{\beta} u_{\beta s}
\]
and
\begin{equation}
	\begin{aligned}
		W_{ij} =& u_{\alpha ij} + \rho_{\alpha ij} u_n + \rho_{\alpha i} u_{nj} + \rho_{\alpha j} u_{ni}\\\notag
		&+ \rho_{\alpha} u_{nij} - \sum_{\beta=1}^{n-1}(u_{\beta i} u_{\beta j} + u_{\beta} u_{\beta ij}).\notag
	\end{aligned}
\end{equation}
%Since the proof of the following lemma is similar to that of Lemma 5.3 in \cite{JaoSun22}, we omit its proof. For reader's convenience, we provide
%a detailed proof for a similar result (Lemma \ref{gj-lem2}) later.
%Similar to Lemma 5.3 in \cite{JaoSun22}, we prove the following lemma:
We also need the following lemma:
\begin{lemma} If $\delta$ is sufficiently small, we have
	\label{BC2-lem1}
	\begin{equation}
		\label{BC2-15}
		LW \leq C \left(1 + |D W| + \sum_i G^{ii} + G^{ij} W_i W_j\right),
	\end{equation}
	where $C$ is a positive constant depending on $n$, $\theta_0$, $\|\psi\|_{C^1 (\ol \Omega \times [-\mu_0, \mu_0]}$ and $\partial \Omega$, where $\mu_0 = \|u\|_{C^0 (\overline{\Omega})}$.
\end{lemma}
The proof of this lemma is similar to lemma 5.5 in \cite{GJ25}, we put it in the appendix for completeness. 

Assume \eqref{BC2-15} is correct, next, we establish the mixed tangential-normal derivative estimate \eqref{BC2-3}.

Note that there exist positive constants $\theta$ and $K$ such that
\begin{equation}
	\label{BC2-5'}
	(\kappa_1' - 3 \theta, \ldots, \kappa_{n-1}' - 3 \theta, 2 K) \in \widetilde{\Gamma}_k
\end{equation}
holds.

Define
\begin{equation}
\label{BC2-6}
v = \rho (x') - x_n - \theta |x'|^2 + K x_n^2
\end{equation}
on $\ol \omega_{\delta}$.
We see that when $\delta(\theta, K)$ is sufficiently small, we have
\begin{equation}
\label{BC2-12}
\begin{aligned}
v \leq & - \frac{\theta}{2} |x'|^2, & \mbox{ on } \partial_1 \omega_\delta,\\
v \leq & - \frac{\delta^2}{2}, & \mbox{ on } \partial_2 \omega_\delta,\\
v \leq & - \frac{\theta \delta^2}{2},   & \mbox{ on } \partial_3 \omega_\delta.
\end{aligned}
\end{equation}
Besides, by \eqref{BC2-1} and \eqref{BC2-5'}, we have   
\begin{equation}
	\label{BC2-5}
\lambda (D^2 v(0))=	(\kappa_1' - 2 \theta, \ldots, \kappa_{n-1}' - 2 \theta, 2 K) \in \widetilde{\Gamma}_k,
\end{equation}
 where $\lambda (D^2 v)$ denotes the eigenvalues of $D^2 v$. 
%and strictly convex for $n=2$.
Thus, there exists an uniform constant $\eta_0 > 0$ depending only on $\theta$, $\p\Omega$
and $K$ satisfying
\[
\lambda (D^2 v - 2 \eta_0 I) \in \widetilde{\Gamma}_k \mbox{ on } \overline{\omega}_\delta.
  %\mbox{ and } \lambda (D^2 v - 2 \eta_0 I) \in \Gamma_{2} \mbox{ for } n=2 \mbox{ on } \overline{\omega}_\delta.
\]
%Then by a similar calculation of Lemma 3.1 in \cite{Ivochkina91} or Proposition 2.1 in \cite{JW22}, we have
We claim,
\begin{equation}
\label{BC2-4}
\lambda \left(\frac{1}{w} \{\gamma^{is} (v_{st} - 2 \eta_0 \delta_{st}) \gamma^{jt}\}\right) \in \widetilde{\Gamma}_k
%\subset \Gamma
\mbox{ on }
\overline{\omega}_\delta.
\end{equation}
%since $\Gamma_{k + 1} \subset \Gamma_k (D u)$.
Actually, we can assume $D^2 v=\operatorname{diag}(\kappa_{1}'-2\theta,\dots,\kappa_{n-1}'-2\theta,2K)$ by rotating the coordinate system if necessary and note that $\gamma^{ij}=\operatorname{diag}(1,\dots,1,\frac{1}{w})$  at $0$ since $|Du(0)|=|u_{n}(0)|$. 
It follows that
\[
\left\{\frac{1}{w}(\gamma^{ik}D_{kl}v\gamma^{lj})\right\}
=\operatorname{diag}\left(\frac{\kappa_{1}'-2\theta}{w},\dots,\frac{\kappa_{n-1}'-2\theta}{w},\frac{2K}{w^3}\right).
\]
Let $\kappa_{v}=(\frac{\kappa_{1}'-2\theta}{w},\dots,\frac{\kappa_{n-1}'-2\theta}{w},\frac{2K}{w^3})=(\kappa_{v}',\frac{2K}{w^3})$. 
By \eqref{BC2-5} and $w<1$, we have $(\kappa_{v}',\frac{2K}{w})\in \widetilde{\Gamma}_k$. 
Then it follows from the ellipticity of $\sigma_k$ in $\widetilde{\Gamma}_{k}$ that $\kappa_{v}\in \widetilde{\Gamma}_k$. 
Hence we can find a suitable constant $\delta$ such that \eqref{BC2-4} holds.  

As \cite{JW22} and \cite{JaoSun22}, we consider the following barrier on $\overline{\omega}_\delta$, for sufficiently small $\delta$,
\begin{equation}
	\label{BC2-20}
	\Psi := v - td + \frac{N}{2} d^2,
\end{equation}
where $v(x)$ is defined by \eqref{BC2-6}, $d (x) := \mathrm{dist} (x, \partial \Omega)$ is the distance from $x$ to the boundary $\partial \Omega$,
and $t,N$ are two positive constants to be determined later. Since $f^{1/(k-l)}$ is concave in $\widetilde {\Gamma}_k$ and homogeneous of degree one and $|Dd| \equiv 1$ on the boundary $\partial \Omega$,  by \eqref{BC2-4} and \eqref{cj-2}, we have,
\[
\begin{aligned}
	&\frac{1}{k-1} \psi^{\frac{1}{k-l} - 1}G^{ij} (D^2 v - \eta_0 I + N D d \otimes D d)_{ij}\\
	\geq \,& G^{1/{(k-l)}} (D^2 v - \eta_0 I + N D d \otimes D d, Du)\\
	\geq \,& \mu (N) \mbox{ on } \overline{\omega}_\delta
\end{aligned}
\]
for some positive constant $\mu (N)$ satisfying $\lim_{N \rightarrow + \infty}\mu (N) = +\infty$.
We then have
\begin{equation}
	\label{gj-15}
	\begin{aligned}
		G^{ij} \Psi_{ij} 
		= \,& G^{ij} v_{ij} - t G^{ij} d_{ij} + N G^{ij} d_i d_j + N d G^{ij} d_{ij}\\
		\geq \,& (k-l) \psi^{1-1/{(k-l)}} \mu(N) + \eta_0 \sum_i G^{ii} + (N d - t) G^{ij} d_{ij}\\
		%\geq \,& n \epsilon_0^{1-1/n}\mu(N) + \eta_0 \sum_i G^{ii} + (N d - t) G^{ij} d_{ij}\\
		\geq \,& 2 \mu_1(N) + (\eta_0 - C N \delta - C t) \sum_i G^{ii}
		%  \geq \,& k B (N) \Big(\frac{\epsilon_0}{2}\Big)^{1-1/k} + \frac{\eta_0}{2} \sum F^{ii}
	\end{aligned}
\end{equation}
on $\overline{\omega}_\delta$, where $\mu_1 (N) := (k-l) (\inf \psi)^{1-1/(k-l)}\mu(N)/2$.
Define
\begin{equation}\label{new3.2}
	\tilde{W} := 1 - \exp\{- b W\}.
\end{equation}
By \eqref{BC2-15}, we can choose the constant $b$ large enough so that
\begin{equation}
	\label{new3.4}
	\begin{aligned}
		L \tilde{W} = \,& G^{ij} \big(-e^{-bW} b^2 W_i W_j + b e^{-bW} W_{ij}\big)\\
		\leq \,& b e^{-bW} \left[ C \left(1 + |D W| + \sum_i G^{ii} \right) + (C - b) G^{ij} W_i W_j\right]\\
		\leq \,&  C (1 + |D \tilde{W}| + \sum_i G^{ii}) + (C - b) G^{ij} W_i W_j b e^{-bW}\\
		\leq \,&  C (1 + |D \tilde{W}| + \sum_i G^{ii}).
	\end{aligned}
\end{equation}
We consider the function
\[
\Phi := R \Psi - \tilde{W},
\]
where $R$ is a large undetermined positive constant. We shall prove
\begin{equation}
	\label{add-6}
	\Phi \leq 0 \mbox{ on } \overline{\omega}_\delta
\end{equation}
by choosing suitable positive constants $\delta$, $t$, $N$ and $R$.

We first consider the case that the maximum of $\Phi$
is achieved at an interior point $x_0 \in \omega_\delta$. It follows that at $x_0$,
\[
|D\tilde{W}|=R |D\Psi|
\]
and if $N$ is sufficiently large and $\delta < \sqrt{\mu_1 (N)}/2CN$,
\[
|D \Psi| = |D v - t D d + N d D d| \leq C (1+t) + C\delta N \leq \mu_1(N)^{1/2} \mbox{ in } \omega_\delta.
\]
Therefore, by \eqref{gj-15}, provided $\delta$ and $t$ sufficiently small such that $C N \delta + C t < \eta_0/2$, we have
\begin{equation}
	\label{BC2-21}
	L \Psi \geq \mu_1 (N) + \mu_1(N)^{1/2} |D \Psi| + \frac{\eta_0}{2} \sum_i G^{ii}.
\end{equation}
By \eqref{new3.4} and \eqref{BC2-21} we obtain, at $x_0$,
\[
\begin{aligned}
	0 \geq L \Phi
	\geq \,& R\mu_1 (N) + R\mu_1(N)^{1/2} |D \Psi| + \frac{R\eta_0}{2} \sum G^{ii}\\
	& - C \left(1 + |D \tilde{W}| + \sum G^{ii}\right)\\
	\geq \,& R \mu_1(N) - C + R (\mu_1(N)^{1/2} - C)|D \Psi|  \\\,&+\left(\frac{R\eta_0}{2} - C\right)\sum G^{ii} > 0
	%> \,&0
\end{aligned}
\]
provided $N$ and $R$ are chosen sufficiently large which is a contradiction. Thus, the maximum of $\Phi$
is achieved at the boundary $\partial \omega_\delta$.
We may further assume $\delta < 2t/N$ so that
\begin{equation}\label{new3.1}
	- t d + \frac{N}{2} d^2 \leq 0 \mbox{ on } \overline{\omega}_\delta.
\end{equation}
By \eqref{BC2-12} and \eqref{new3.1},
we can conclude that $\Phi \leq 0$ on $\partial \omega_\delta$ by choosing $R$ larger and then \eqref{add-6} is proved.

Since $(R \Psi - \tilde{W}) (0) = 0$, we have $(R \Psi - \tilde{W})_n (0) \leq 0$. Therefore, we get
\[
u_{n \alpha} (0) \geq - C.
\]
The above arguments also hold for
\[
W = - \nabla'_\alpha u - \frac{1}{2} \sum_{1 \leq \beta \leq n - 1} u_\beta^2.
\]
Hence, we obtain \eqref{BC2-3}.

\noindent
{\bf Part 3.} The double normal derivative estimate on $\partial \Omega$, that is 
\begin{equation}
	\label{gj-16}
	|u_{nn} (0)| \leq C.
\end{equation}

At the origin, we may assume $\{u_{\alpha\beta}\}_{1\leq\alpha,\beta\leq n-1}$ is diagonal by a
rotation of axes if necessary. Note that 
\[
g^{ij}=\delta_{ij}+\frac{|Du|^2}{w^2}\delta_{in}\delta_{jn}
\]
at the origin, we have 
\[
a_{ij}=
\begin{pmatrix}
	\frac{u_{11}}{w} & 0 & \cdots & 0 & \frac{u_{1n}}{w} \\
	0 & \frac{u_{22}}{w} & \cdots & 0 & \frac{u_{2n}}{w} \\
	\vdots & \vdots & \ddots & \vdots & \vdots \\
	0 & 0 & \cdots & \frac{u_{n-1, n-1}}{w} & \frac{u_{n-1,n}}{w} \\
	\frac{u_{n1}}{w} & \frac{u_{n2}}{w} & \cdots & \frac{u_{n,n-1}}{w} & \frac{u_{nn}}{w^3}
\end{pmatrix}
\]
By Lemma 1.2 of \cite{CNSIII} and the estimates of $u_{\alpha\beta}(0)$, $u_{\alpha n}(0)$, there exists a constant $R_1 > 0$ sufficiently large such that if $u_{nn} > R_1$, then
\[
\begin{cases}
	\kappa_\alpha = \dfrac{u_{\alpha\alpha}}{w} + o(1), & \alpha = 1, \cdots, n - 1, \\[10pt]
	\kappa_n = \dfrac{u_{nn}}{w^3} + O(1),
\end{cases}
\]
where $\kappa_{1}, \ldots, \kappa_{n}$ are the eigenvalues of $\{a_{ij}\}$.

Since $k < n$, by choosing $R_{1}$ large enough,  we have at the origin, 
\begin{equation}\label{homo-case}
	\psi=\Big(\frac{\sigma_{k}}{\sigma_{l} }\Big) (\lambda(\eta))
	= \frac{C_{n-1}^k u_{nn}^k/(w^{3k}) +  O(u_{nn}^{k-1})}{C_{n-1}^l u_{nn}^l/(w^{3l}) + O(u_{nn}^{l-1})}. 
\end{equation}
It then implies the uniform upper bound of $|u_{nn}(0)|$. 
Combining \eqref{BC2-2} and \eqref{BC2-3}, we can obtain the uniform upper bound of $|u_{ij}(0)|$. Then \eqref{gj-16} follows immediately  and Theorem \ref{thm-boundary} is proved.

The last step is to estimate the second derivatives in the interior.

We need the following lemma of Urbas\cite{Urbas 03}.  As usual, unless otherwise specified, we assume  that the same indicator represents the summation.
\begin{lemma}
	\label{lem-U}
	The second fundamental form $h_{ab}$ satisfies
	\begin{equation}
		\label{GJ4.0}
		\begin{aligned}
			F^{ij}\nabla_{ij}h_{ab}=\,&-F^{ij,kl}\nabla_{a}h_{ij}\nabla_{b}h_{kl}-F^{ij}h_{ij}h_{ak}h_{bk}\\
			\,&+F^{ij}h_{ik}h_{jk}h_{ab}+\nabla_{ab}\psi
		\end{aligned}
	\end{equation}
	where
	\[
	F^{ij}=\frac{\partial f(\lambda(h))}{\partial h_{ij}}.
	\]
\end{lemma}

Suppose the maximum of $h_{\xi\xi}$ is
achieved at $X_0 = (x_0, u (x_0))\in M_u$ and $\xi_0 \in T_{X_0} M_u$.
We take point $X_{0}$ as the origin and establish a  local orthonormal frame $\{e_{1},e_{2},\cdots,e_{n}\}$ such that,
\[
\xi_0 = e_1,\ \nabla_{e_i} e_j = 0 \mbox{ at } X_0.
\]
We may also assume that $\{h_{ij}\}$ is diagonal at $X_0$ and furthermore,
\[
h_{11} \geq \cdots \geq h_{nn}
\]
by rotating the coordinate system if necessary.

Using \eqref{GJ4.0} and differentiating the equation \eqref{GJ1.1} twice, we have
\begin{equation}
	\label{GJ-11}
	0\geq\tilde{F}^{ii} \nabla_{ii}h_{11} = -\tilde{F}^{ij,kl}\nabla_{1}h_{ij}\nabla_{1}h_{kl} - \tilde{\psi} h_{11}^2 + h_{11} \tilde{F}^{ii} h_{ii}^2 + \nabla_{11} (\tilde{\psi}).
\end{equation}
where $\tilde{\psi}=\psi^{1/(k-l)}$ and
\[
\tilde{F}^{ij} = \frac{\partial (\frac{\sigma_{k}}{\sigma_{l}})^{1/(k-l)} (\lambda (h))}{\partial h_{ij}},\;\;\;
\tilde{F}^{ij,kl} = \frac{\partial^2 (\frac{\sigma_{k}}{\sigma_{l}})^{1/(k-l)} (\lambda (h))}{\partial h_{ij}\partial h_{kl}}.
\]
Thus, it follows from \eqref{ile} and the concavity of $(\frac{\sigma_{k}}{\sigma_{l}})^{1/(k-l)}(\lambda (h))$ that
\[
0\geq c(n,k,l) \psi^{1 - 1/{(k-l)}} h_{11}^3 - \tilde{\psi} h_{11}^2 + \nabla_{11} (\tilde{\psi})\geq  c(n,k,l)\psi^{1 - 1/{(k-l)}} h_{11}^3 - Ch_{11}^2
\]
provided $h_{{11}}$ large enough. Then a bound for $h_{11}$ is derived from the above inequality.

Combined with the boundary estimates of the second-order derivatives \eqref{B2-0}, we have 
\begin{equation}
	\label{B2-01}
	\max_{\ol \Omega} |D^2 u| \leq C.
\end{equation}
We have done.

\section{$C^2$ estimates for general boundary data}
For this part, it suffices to adjust the proof for the boundary in Section 5.

\noindent
{\bf Part 1.} We now differentiate the boundary condition $u = \varphi$ twice on $\partial \Omega$ to derive the double-tangent estimates. 

\noindent
{\bf Part 2.} In the mixed tangential-normal derivative estimate, we have used the fact that
$\gamma^{ij}=\operatorname{diag}(1,\dots,1,\frac{1}{w})$  at an fixed boundary point since $|Du(0)|=|u_{n}(0)|$, which is not true for $\varphi\neq$ constant. 
Here we shall illuminate \eqref{BC2-4} for general $\varphi$.

Let $A=\{a_{ij}\}$ be a matrix with $|a_{nn}|\rightarrow +\infty$ and other terms bounded.
Given a boundary point %$x_{0}\in \partial \Omega$, 
$x_{0}\in \partial\Omega$,
we have 
\[
g^{-1}=\{\delta_{ij}+\frac{u_{i}u_{j}}{w^2}\}
\]
and
\[
g^{-1}\frac{A}{w}=\frac{1}{w}
\begin{pmatrix}
	o(|a_{nn}|)   & \cdots                   & o(|a_{nn}|)            & \frac{u_{1}u_{n}}{w^2}a_{nn}+o(|a_{nn}|) \\
	\vdots                       & \ddots                       & \ddots        & \frac{u_{2}u_{n}}{w^2}a_{nn}+o(|a_{nn}|)                  \\
	\vdots                  & \ddots           & \ddots           & \vdots \\
	\vdots                 & \ddots       & o(|a_{nn}|)              & \frac{u_{n-1}u_{n}}{w^2}a_{nn}+o(|a_{nn}|) \\
	o(|a_{nn}|)  & \cdots                       & o(|a_{nn}|)             & (1+\frac{u_{n}^{2}}{w^2})a_{nn}+o(|a_{nn}|)
\end{pmatrix},
\]
which means all entries of $g^{-1}\frac{A}{w}$ are $o(|a_{nn}|)$, except those in the $n$-th column.

Consider the matrix $B:=\operatorname{tr}(g^{-1}\frac{A}{w}) I-g^{-1}\frac{A}{w}$, we have
\[
B=
\frac{1}{w}
\begin{pmatrix}
	(1+\frac{u_{n}^{2}}{w^2})a_{nn}+o(|a_{nn}|)   & \cdots                   & o(|a_{nn}|)            & -\frac{u_{1}u_{n}}{w^2}a_{nn}+o(|a_{nn}|) \\
	\vdots                       & \ddots                       & \ddots        & -\frac{u_{2}u_{n}}{w^2}a_{nn}+o(|a_{nn}|)                  \\
	\vdots                  & \ddots           & \ddots           & \vdots \\
	\vdots                 & \ddots       &  (1+\frac{u_{n}^{2}}{w^2})a_{nn}+o(|a_{nn}|)              & -\frac{u_{n-1}u_{n}}{w^2}a_{nn}+o(|a_{nn}|) \\
	o(|a_{nn}|)  & \cdots                       & o(|a_{nn}|)             & o(|a_{nn}|)
\end{pmatrix}.
\]
By standard mathematical induction, we have
\[
B^{k}=
\frac{1}{w^{k}}
\begin{pmatrix}
	[(1+\frac{u_{n}^{2}}{w^2})a_{nn}]^{k}+o_{k}   & \cdots                   & o_{k}            & -(1+\frac{u_{n}^{2}}{w^2})^{k-1}\frac{u_{1}u_{n}}{w^2}a_{nn}^{k}+o_{k} \\
	\vdots                       & \ddots                       & \ddots        & -(1+\frac{u_{n}^{2}}{w^2})^{k-1}\frac{u_{2}u_{n}}{w^2}a_{nn}^{k}+o_{k}                  \\
	\vdots                  & \ddots           & \ddots           & \vdots \\
	\vdots                 & \ddots       &  [(1+\frac{u_{n}^{2}}{w^2})a_{nn}]^{k}+o_{k}              & -(1+\frac{u_{n}^{2}}{w^2})^{k-1} \frac{u_{n-1}u_{n}}{w^2} a_{nn}^{k}+o_{k} \\
	o_{k}  & \cdots                       & o_{k}             & o_{k}
\end{pmatrix},
\]
where $o_{k}=o(|a_{nn}|^{k})$. 
Thus $\operatorname{tr}(B^k)=(n-1)[(1+\frac{u_{n}^{2}}{w^2})a_{nn}]^{k}+o_{k}$.

Note that $\operatorname{tr}(B^k)=p_k$, by Newton's Formulas we have
\[
\sigma_{1}=p_1=(n-1)\frac{(w^2+u_{n}^2)a_{nn}}{w^{3}}+o_{1}:=(n-1)a+o_{1}
\]
and
\[
\sigma_{2}=\frac{1}{2}(\sigma_{1}p_1-p_2)=\frac{(n-1)(n-2)}{2}a^2+o_{2}=C^{2}_{n-1}a^{2}+o_{2}.
\]
Again by standard mathematical induction, for $k<n$, we have 
\begin{equation}\label{alg-str}
	\sigma_{k}=C^{k}_{n-1}a^{k}+o_{k}=C_{n-1}^{k}\frac{1}{w^{3k}}(w^2+u_{n}^2)^{k}a_{nn}^{k}+o(|a_{nn}|^{k}).
\end{equation}
Since \(w^2+u_n^2>0\) from the gradient estimates, we conclude that \eqref{BC2-4} holds by applying \(A=D^2v-2\eta_0I\).

Similarly, we can derive the estimates for mixed tangential-normal derivatives on the boundary.

\noindent
{\bf Part 3.}
We now only need to establish the second order double-normal estimates on the boundary.
By the same argument applied to \(A=D^2u\) in Part 2, we have
\[
\sigma_{k}=C_{n-1}^{k}\frac{1}{w^{3k}}(w^2+u_{n}^2)^{k}u_{nn}^{k}+o(|u_{nn}|^{k}).
\] 
Thus the algebraic structure of $\eta$-type curvature quotient operator ($k<n$) with general boundary data is
\begin{equation}\label{general-case}
	\psi=\Big(\frac{\sigma_{k}}{\sigma_{l} }\Big) (\lambda(\eta))
	= \frac{C_{n-1}^k (w^2+|u_n|^2)^{k}u_{nn}^k/(w^{3k}) +  O(|u_{nn}|^{k-1})}{C_{n-1}^l (w^2+|u_n|^2)^{l} u_{nn}^l/(w^{3l}) + O(|u_{nn}|^{l-1})},  
\end{equation}
then we have finished the proof, since $w^2+u_{n}^2>0$ by the gradient estimates.

\section{Proof of the main result}

The proof of Theorem \ref{GJ-thm1} relies on a standard continuity method (see \cite{CNSV}). 
%In Section 3, we construct a spacelike function $\ul u$ satisfying \[\left\{\begin{aligned}	&\frac{\sigma_{k}}{\sigma_{l}} \left( \lambda (\eta [M_{\ul u}]) \right) = \psi^{0}(x,\ul u) \geq \psi(x, u) &&\mbox{in}~\Omega, \\	&\ul u = 0 &&\mbox{on}~\partial\Omega. \end{aligned}\right.\]
By assumption, $\varphi$ has an admissible extension, we still denote it by $\varphi$. 
Let $$\psi^{0}(x,\varphi) = \frac{\sigma_k}{\sigma_l} \left( \lambda (\eta [M_{\varphi}]) \right)$$ and extend $\psi^{0}$ to $\Omega\times \mathbb{R}$ with $\psi^{0}(x,z)=\psi^{0}(x,\varphi)$.
Extend $\psi$ in the same way, 
 we consider, for each \( t \in [0,1] \),
\begin{equation}
	\label{continuity}
	\left\{
	\begin{aligned}
		&\frac{\sigma_k}{\sigma_l} \left( \lambda (\eta [M_{u^t}]) \right) = t \psi + (1-t)\psi^0 &&\mbox{in}~\Omega,\\
		&u^t= \varphi &&\mbox{on}~\partial \Omega.
	\end{aligned}
	\right.
\end{equation}
The set \( S \) consists of numbers \( t \in [0,1] \) such that there exists a function \( u^t \in C^{2}(\overline{\Omega}) \), which is an admissible solution of \eqref{continuity}.

Obviously, $\varphi$ is the solution of \eqref{continuity} for $t=0$, so the set \( S \) is nonempty. 
We shall prove that $S$ is both open and close, so $S=[0,1]$ and the proof is finished. 
First, by the implict function theorem, we can obtain $S$ is an open set. Next, let \( \{s_i\}_{i \in \mathbb{N}} \) be an arbitrary sequence from the set \( S \) and \( s \in [0,1] \) satisfies \( \lim_{i \to \infty} s_i = s \), we shall prove \( s \in S \).
Combine with the estimates
\[
     \sup_{\overline{\Omega}}|D u^t| \leq 1 - \theta_0 <1
\]
and
\[
	\| u^t \|_{C^2(\overline{\Omega})} \leq C,  
\]
we obtain the uniform ellipticity of $\frac{\sigma_{k}}{\sigma_l} (\lambda(\eta))$. 

Since $(\frac{\sigma_{k}}{\sigma_l} (\lambda(\eta)))^{\frac{1}{k-l}}$ is concave (see \cite{S05}) in $\widetilde{\Gamma}_{k}$, then by Evans--Krylov theorem, we can obtain a $C^{2,\alpha}$ estimate, %(see \cite{TrudingerBook}). 
which means
\[
\| u^{s_i} \|_{C^{2,\alpha}(\overline{\Omega})} \leq C, 
\]
where \( C \) is independent on \( s_i \). By the Arzela-Ascoli theorem, we have \( u^s\in C^{2,\delta}(\overline{\Omega}) \) for \( \delta \in (0, \alpha) \).
Therefore, \( s \in S \) and \( S \) is closed. 

In addition, by standard nonlinear elliptic regularity, we can obtain \( u \in C^\infty(\overline{\Omega}) \).

The uniqueness of the solution follows from the maximum principle.

%\emph
%\begin{proof}[\bf Proof of theorem 1.3]	Combined \eqref{jg-7'}, \eqref{spacelike}, \eqref{B2-01}, we can obtain	\begin{equation}		\label{B2-011}		||u||_{C^2(\ol{\Omega})} \leq C.	\end{equation}Applying the Evans-Krylov theory, we can obtain the $C^{2,\alpha}$ estimates for $u$.Using the continuity method and the maximum principle, we can finally establish the existence and uniqueness of solutions to \eqref{GJ1.3}.\end{proof}

%\begin{remark}
%We emphasize that in the case $k<n$, the suitable function $\underline{\psi}$ in \cite{JS22} can be ignored to derive boundary double-normal second order derivative estimates.
%In \cite{JS22}, Jiao-Sun need the existence of a suitable function $\underline{\psi}(x,z)\in C^0(\overline{\Omega}\times\mathbb{R})$ to prove that $-u_n(0)$ has a positive lower bound, which will used to derive boundary normal-normal second order derivative estimates. But we do not need the existence of $\underline{\psi}$ in Theorem \ref{main} since we only consider the case that $k<n$.
%\end{remark}

\section*{Appendix}

\subsection*{Proof of Lemma \ref{BC2-lem1}}

\;

%The proof of lemma \ref{BC2-lem1} is similar to those of Lemma 5.5 in  \cite{GJ25} and Lemma 5.2 in \cite{CTX25}, for reader's convenience, we provide a detailed proof here.

By differentiating the equation \eqref{1-1-1}, we have
	\begin{equation}
		\label{BC2-16.}
		\begin{aligned}
			L W + G^s W_s
			\leq \,& C + C \sum_i G^{ii} + C \sum_s G^{s} + 2 G^{ij} u_{ni} \rho_{\alpha j} \\
			&%- 2 \sum_{\beta \leq n-1} G^{ij} u_{\beta i} \varphi_{\beta j} 
			- \sum_{\beta \leq n-1} G^{ij} u_{\beta i} u_{\beta j}.
		\end{aligned}
	\end{equation}
	By
	\[
	a_{ij} = \frac{1}{w}\gamma^{ik}u_{kl} \gamma^{lj},
	\]
	we can obtain
	\[
	G^{ij} = \frac{\partial G}{\partial u_{ij}} = F^{st}\frac{\partial a_{st}}{\partial u_{ij}} = \frac{1}{w}\sum_{s,t} F^{st} \gamma^{is} \gamma^{tj}
	\] 
	%\mbox{ and }
	and
	\[
	u_{ij} = w\sum_{s,t} \gamma_{is} a_{st} \gamma_{tj}.
	\]
	It follows that
	\[
	\sum_{\beta \leq n - 1} G^{ij} u_{\beta i} u_{\beta j}
	= w \sum_{\beta \leq n - 1} \sum_{s,t}F^{ij} \gamma_{\beta s} \gamma_{\beta t} a_{si} a_{tj}
	\]
	and
	\[
	\frac{1}{w} \sum_{i} F^{ii}\leq\sum_{i}G^{ii}\leq \frac{1}{w^3} \sum_{i} F^{ii}.
	\]
	By \cite{GS04}, we can find an orthogonal matrix $B = (b_{ij})$ which can diagonalize $(a_{ij})$ and $(F^{ij})$ at the same time:
	\[
	F^{ij} =\sum_s b_{is} f_s b_{js} \mbox{ and } a_{ij} =\sum_s b_{is} \kappa_s b_{js}.
	\]
	Therefore, we arrive
	\begin{equation}
		\begin{aligned}
			\sum_{\beta \leq n - 1} G^{ij} u_{\beta i} u_{\beta j} =& w \sum_{\beta \leq n - 1} F^{ij} \gamma_{\beta s} \gamma_{\beta t} a_{si} a_{tj}\\\notag	
			=& w \sum_{\beta \leq n - 1} \sum_i\left(\sum_s\gamma_{\beta s} b_{si}\right)^2  f_i \kappa_i^2\notag.
		\end{aligned}
	\end{equation}
	Let the matrix $\eta = (\eta_{ij}) = (\sum_s\gamma_{is} b_{sj})$. We can easy to know $\eta \cdot \eta^T = g$ and $|\det (\eta)| = \sqrt{1 - |Du|^2}$.
	Hence, we obtain
	\begin{equation}
		\label{BC2-14.}
		\sum_{\beta \leq n - 1} G^{ij} u_{\beta i} u_{\beta j}
		= w \sum_{\beta \leq n - 1} \sum_i\eta_{\beta i}^2  f_i \kappa_i^2.
	\end{equation}
	We have
	\begin{equation}
		\label{BC2-19.}
		G^{ij} u_{ni} \rho_{\alpha j} = \sum_{i,t}f_i \kappa_i b_{si} \gamma^{js} b_{ti} \gamma_{nt} \rho_{\alpha j} \leq C \sum_i f_i |\kappa_i|.
	\end{equation}
	%	and
	%	\begin{equation}
		%		\label{BC2-19.1.}
		%		G^{ij} u_{\beta i} \varphi_{\beta j}= \sum_{i,t}f_i \kappa_i b_{si} \gamma^{js} b_{ti} \gamma_{ \beta t} \varphi_{\beta j} \leq C \sum_i f_i |\kappa_i|.
		%	\end{equation}
	For any indices  $j, t$, we have
	\[
	F^{ij} a_{it} =\sum_{i,s,p} b_{is} f_s b_{js} b_{ip} \kappa_p b_{tp} = \sum_i f_i \kappa_i b_{ji} b_{ti}.
	\]
	Thus,
	by \eqref{GS-2}, we find
	\begin{equation}
		\label{BC2-11.}
		\begin{aligned}
			\mid \sum_s G^s \mid =& \bigg| \sum_s \biggl\{ \frac{u_s}{w^2} \sum_i f_i \kappa_i + \frac{2}{w (1+w)} \sum_{t,j}F^{ij} a_{it} \big(w u_t \gamma^{sj} 
			+ u_j \gamma^{ts}\big) \biggr\} \bigg|\\\notag
			\leq&  C \sum_i f_i |\kappa_i|.
		\end{aligned}
	\end{equation}
	Combining \eqref{BC2-16.}-\eqref{BC2-11.},
	we obtain
	\begin{equation}
		\label{BC2-26.}
		\begin{aligned}
			LW + G^s W_s
			\leq C \left(1 + \sum_i G^{ii} + \sum_i f_i |\kappa_i|\right) - w \sum_{\beta \leq n - 1} \sum_i\eta_{\beta i}^2 f_i \kappa_i^2.
		\end{aligned}
	\end{equation}
	Now we consider the term $G^s W_s$. We have, by \eqref{GS-2} and the definition of the matrix $(b_{ij})$,
	\begin{equation}
		\label{BC2-17.}
		\begin{aligned}
			- G^s W_s = \,& -\frac{u_s}{w^2} \sum_i f_i \kappa_i W_s - \frac{2}{w (1+w)} \sum_{t,j}F^{ij} a_{it} \big(w u_t \gamma^{sj} 
			+ u_j \gamma^{ts}\big) W_s\\
			= \,& -\frac{1}{w} \sum_s\left((k-l) \psi u_s\frac{1}{w} + 2 \sum_{t,i}f_i \kappa_i (b_{ti} u_t) \gamma^{sl} b_{li}\right) W_s\\
			\leq \,& C |D W| - \frac{2}{w} \sum_{t,i}f_i \kappa_i (b_{ti} u_t) \gamma^{sl} b_{li} W_s.
		\end{aligned}
	\end{equation}
	We further discuss in two situations: (a) $\sum_{\beta \leq n - 1} \eta_{\beta i}^2 \geq \epsilon$ for all $i$; 
	and (b) $\sum_{\beta \leq n - 1} \eta_{\beta r}^2 < \epsilon$ for some index $1 \leq r \leq n$, where
	$\epsilon$ will be chosen later.
	
	{\bf Case (a)}. 
	By \eqref{BC2-14.}, we have
	\[
	\sum_{\beta \leq n - 1} G^{ij} u_{\beta i} u_{\beta j}
	\geq \epsilon w \sum_{i} f_i \kappa_i^2.
	\]
	By Cauchy-Schwarz inequality, for any $\epsilon_0 > 0$, we have
	\begin{equation}
		\label{BC2-31.}
		\frac{2}{w} \kappa_i (b_{ti} u_t) \gamma^{sl} b_{li} W_s
		\geq -\frac{\epsilon_0}{2} \kappa_i^2 - \frac{C}{\epsilon_0 } (\gamma^{sl} b_{li} W_s)^2.
	\end{equation}
	Then
	\[
	- G^s W_s \leq C |D W| + \frac{\epsilon_0}{2} f_i \kappa_i^2 + \frac{C}{\epsilon_0} G^{ij} W_i W_j.
	\]
	Obviously, for any $\epsilon_1 > 0$,
	$$\sum_i f_i |\kappa_i|\leq \frac{1}{2\epsilon_1} \sum_i f_i +  \frac{\epsilon_1}{2} \sum_i f_i \kappa_i^2
	\leq C\left(\frac{1}{\epsilon_1} \sum_i G^{ii}+ \epsilon_1 \sum_i f_i \kappa_i^2\right).$$
	%Let $\epsilon>\frac{\epsilon_0+\epsilon_1}{2}$,
	Combining the previous four inequalities with \eqref{BC2-26.}, \eqref{BC2-15} follows.

	{\bf Case (b)}.
	By lemma 4.3 of Bayard \cite{Bayard03} , 
	for any $i\neq r$,
	$$\sum_{\beta\leq n-1}\eta^2_{\beta i}\geq c_1$$ for some positive constant $c_1$ depending on $\|u\|_{C^1 (\overline{\Omega})}$ and $n$. 
	
	Then, in view of \eqref{BC2-14.},
	\begin{equation}
		\label{BC2-13.}
		\sum_{\beta \leq n - 1} G^{ij} u_{\beta i} u_{\beta j} \geq w c_1 \sum_{i \neq r} f_i  \kappa_i^2.
	\end{equation}
	
	If $\kappa_r \leq 0$, by Lemma 2.7 of \cite{Guan14}, we have
	\[
	\sum_{i \neq r} f_i  \kappa_i^2
	\geq \frac{1}{n+1} \sum_{i = 1}^n f_i \kappa_i^2.
	\]
	Thus \eqref{BC2-15} follows using a similar argument as the Case (a).
	
	So next, we may assume $\kappa_r > 0$. Without loss of generality, assume that $\kappa_1 \geq \cdots \geq \kappa_n$. Let  $\lambda = (\lambda_1, \cdots, \lambda_n)$ be the eigenvalues of $\eta[M_u]$, then $\lambda_i = \sum_{j \neq i} \kappa_j$. It follows that $\lambda_1 \leq \cdots \leq \lambda_n$. Then we consider two subcases.
	
	\noindent
	{\bf Case (b-1).}
	$\kappa_n \geq - \epsilon_0 \kappa_r$, where $\epsilon_0$ is a positive constant to be chosen later.
	
	In order to derive \eqref{BC2-15}, the key is to prove the important inequality
	\begin{equation}\label{claim}
		f_r \kappa_r \leq C.
	\end{equation}
	
	\emph{Proof of \eqref{claim}:}
	If $\kappa_n \geq -\epsilon_0 \kappa_r$, then for any $i \neq r$, by choosing $\epsilon_0 \leq \frac{1}{2(n-2)}$ we derive
	\begin{equation}\label{b21}
		\lambda_i = \sum_{j \neq i}\kappa_j \geq[1-(n-2)\epsilon_0] \kappa_r \geq \frac{1}{2} \kappa_r > 0.
	\end{equation}
	
	When $r \neq 1$, it is obvious that
	$$\lambda_n \geq \cdots \geq \lambda_1 \geq[1-(n-2) \epsilon_0] \kappa_r \geq \frac{1}{2} \kappa_r > 0.$$
	Thus
	\begin{equation}
		\begin{aligned}
			\label{fr}
			f_r \kappa_r & = \sum_{p \neq r}
			\frac{\sigma_{k-1; p}(\eta) \sigma_l(\eta) - \sigma_k(\eta) \sigma_{l-1; p}(\eta)}{\sigma_l(\eta)^2} \kappa_r\\
			& \leq (n-k+1) \frac{\sigma_{k-1}(\eta)}{\sigma_l(\eta)} \kappa_r\\
			& \leq \frac{(n-k+1) C_n^{k-1}}{\sigma_l(\eta)} \lambda_n \cdots \lambda_{n-k+2} \kappa_r\\
			& \leq \frac{2(n-k+1) C_n^{k-1}}{\sigma_l(\eta)} \lambda_n \cdots \lambda_{n-k+2} \lambda_{n-k+1}\\
			& \leq C \frac{\sigma_k(\eta)}{\sigma_l(\eta)}\\ 
			& = C \psi.
		\end{aligned}
	\end{equation}
	When $r = 1$, i.e., $\kappa_n \geq - \epsilon_0 \kappa_1$. By \eqref{b21},
	$$\lambda_i \geq [1-(n-2) \epsilon_0] \kappa_1 > 0, \quad \forall ~ i \geq 2.$$
	Next we consider two cases: Case $\lambda_1 \geq 0$ and Case $\lambda_1 < 0$. 
	
	If $\lambda_1 \geq 0$, similar to \eqref{fr}, it is easy to derive $f_r \kappa_r \leq C \psi$. 
	
	If $\lambda_1 < 0$, then
	$$0 > \lambda_1 =\sum_{i \geq 2} \kappa_i \geq (n-1) \kappa_n \geq - (n-1) \epsilon_0 \kappa_1.$$
	By $\lambda_1 < 0$, we get $\sigma_{k-1}(\eta) \leq \sigma_{k-1; 1}(\eta)$. Thus
	\begin{equation}
		\begin{aligned}
			\label{kap}
			f_1 \kappa_1 & \leq (n-k+1) \frac{\sigma_{k-1}(\eta)}{\sigma_l(\eta)} \kappa_1 \\
			& \leq (n-k+1) \frac{\sigma_{k-1; 1}(\eta)}{\sigma_l(\eta)} \kappa_1\\
			& \leq \frac{(n-k+1) C_{n-1}^{k-1}}{[1-(n-2) \epsilon_0] \sigma_l(\eta)} \lambda_n \cdots \lambda_{n-k+2} \lambda_{n-k+1}.
		\end{aligned}
	\end{equation}
	
	Note that
	\begin{equation}
		\begin{aligned}
			\label{sig}
			\sigma_k(\eta) & = \sigma_{k; 1}(\eta) + \lambda_1 \sigma_{k-1; 1}(\eta)\\
			& \geq \lambda_n \cdots \lambda_{n-k+1} + \lambda_1 \cdot  C_{n-1}^{k-1} \lambda_n \cdots \lambda_{n-k+2}\\
			& = (\lambda_{n-k+1} + C_{n-1}^{k-1} \lambda_1) \lambda_n \cdots \lambda_{n-k+2}.
		\end{aligned}
	\end{equation}
	and
	$$\lambda_{n-k+1} \geq [1 - (n-2) \epsilon_0] \kappa_1 \geq - \frac{1-(n-2) \epsilon_0}{(n-1) \epsilon_0} \lambda_1,$$
	since $n-k+1 \geq 2$. Therefore by choosing $\epsilon_0 \leq \frac{1}{n - 2 + 2 (n-1) C_{n-1}^{k-1}}$ sufficiently small, we get
	\begin{equation}\label{12}
		\frac{1}{2}\lambda_{n-k+1} + C_{n-1}^{k-1} \lambda_1 \geq - \lambda_1 \left(\frac{1 - (n-2) \epsilon_0}{2 (n-1) \epsilon_0} - C_{n-1}^{k-1}\right) \geq 0.
	\end{equation}
	Which implies $\sigma_{k}(\eta) \geq \frac{1}{2} \lambda_n \cdots \lambda_{n-k+1}$. Combining with \eqref{kap}-\eqref{12}, \eqref{claim} still holds.
	
	By \eqref{BC2-17.} and \eqref{BC2-31.} and the Cauchy-Schwarz inequality, we have, for any $\epsilon > 0$,
	\begin{equation}
		\label{BC2-30.}
		\begin{aligned}
			- G^s W_s \leq \,& C |DW| - \frac{2}{w} \sum_{i \neq r} f_i \kappa_i (b_{ti} u_t) \gamma^{sl} b_{li} W_s\\
			\leq \,& C |DW| + C \sum_{i \neq r} f_i |\kappa_i| |\gamma^{sl} b_{li} W_s| \\
			\leq \,& C |DW| + \epsilon \sum_{i \neq r} f_i \kappa_i^2
			+ \frac{C}{\epsilon} \sum_{i=1}^n f_i \gamma^{sl} b_{li} W_s \gamma^{tk} b_{ki} W_t\\
			\leq \,& C |DW| + \epsilon \sum_{i \neq r} f_i \kappa_i^2 + \frac{C}{\epsilon} G^{ij} W_i W_j.
		\end{aligned}
	\end{equation}
	By using \eqref{BC2-13.} and  fixing $\epsilon$ sufficiently small, we can prove \eqref{BC2-15}.

	\noindent
	{\bf Case (b-2).}  $\kappa_n < - \epsilon_0 \kappa_r$.
	
	Since $\kappa_r > 0$, it is obvious that $r \neq n$.
	In this subcase, we get
	\begin{equation}
		\begin{aligned}
			\gamma^{sj} b_{jr} W_s & = \gamma^{sj} b_{jr} (u_{\alpha s} + \rho_{\alpha s} u_n + \rho_{\alpha} u_{ns} - \sum_{\beta\leq n-1} u_{\beta} u_{\beta s}) \\\notag
			& = w \kappa_r (\eta_{\alpha r} + \rho_{\alpha} \eta_{nr} - \sum_{\beta \leq n-1} u_{\beta} \eta_{\beta r}) + \gamma^{sj} b_{jr} \rho_{\alpha s} u_n\notag.
		\end{aligned}
	\end{equation}
	It follows that
	\begin{equation}
		\label{fol}
		|\gamma^{sj} b_{jr} W_s| \leq C w \kappa_r (\epsilon + |\rho_{\alpha}|) + C.
	\end{equation}
	It is obvious that $f_r \kappa_r = (k-l) \psi - \sum_{i\neq r} f_i \kappa_i$, thus by $\kappa_n^2 > \epsilon_0^2 \kappa_r^2$, 
	\begin{equation}
		\begin{aligned}
			\frac{2}{w} f_r \kappa_r |(\sum_t b_{tr} u_t) \gamma^{sj} b_{jr} W_s| 
			= \,& \frac{2}{w}((k-l) \psi - \sum_{i \neq  r} f_i \kappa_i) \left|\left( \sum_t b_{tr} u_t \right) \gamma^{sj} b_{jr}W_s \right|\\\notag
			\leq \,& C |DW| + C (\epsilon + |\rho_{\alpha}|)\sum_{i \neq r } f_i |\kappa_i| \kappa_r + C \sum_{i \neq r} f_i|\kappa_i|\\\notag
			\leq \,& C |DW| + C \epsilon_0^{-1} (\epsilon + |\rho_{\alpha}|) \sum_{i \neq r } f_i \kappa_i^2\\\notag
			& +  C \epsilon_0 (\epsilon + |\rho_{\alpha}|) \sum_{i \neq r } f_i \kappa_r^2 + C \epsilon_1 \sum_{i \neq r} f_i \kappa_i^2 + \frac{C}{\epsilon_1} \sum_{i \neq r} f_i\\\notag
			\leq \,& C |DW| + (C \epsilon_0^{-1} (\epsilon + |\rho_{\alpha}|) + \epsilon_1) \sum_{i \neq r} f_i \kappa_i^2 + \frac{C}{\epsilon_1} \sum_i G^{ii}\notag
		\end{aligned}
	\end{equation}
	for any $\epsilon_1>0$. Here the last inequality comes from $\sum_{i\neq r}f_i\kappa_r^2\leq \frac{n-1}{\epsilon_0^2}f_n\kappa_n^2\leq \frac{n-1}{\epsilon_0^2}\sum_{i\neq r}f_i\kappa_i^2$ with $r\neq n$.

	We can choose sufficiently small $\delta$, $\epsilon$ and $\epsilon_1$ satisfying
	\[
	\left(\epsilon_0^{-1} C (\epsilon + |\rho_\alpha|) + \epsilon_1\right) < \frac{m c_1}{4},
	\]
	where $m=1-\theta_0^2$, $\theta_0$ is defined in spacelike. 
	Therefore, as in Case (b-1), \eqref{BC2-15} follows by \eqref{BC2-13.}.

\bigskip

{\bf Acknowledgment.} We would like to thank Rirong Yuan for pointing out that $\tilde{\Gamma}_k$ is of type 2 when $k<n$. We also thank Heming Jiao and Zhizhang Wang for their helpful comments and valuable suggestions.
 
{\bf Conflict Of Interest Statement.} The authors declare that there is no conflict of interest.

{\bf Data Availability Statement.} The authors declare that this study only contains theoretical derivations and mathematical proofs, with no data generated or analyzed.


\begin{thebibliography}{99}

%\bibitem{Andrews94}
%B. Andrews,
%{\em Contraction of convex hypersurfaces in Euclidean space},
%Calc. Var. PDE  {\bf 2} (1994), 151--171.

\bibitem{Bartnik82} R. Bartnik and L. Simon, {\em Spacelike hypersurfaces with prescribed boundary values and mean curvature}, Comm. Math. Phys. {\bf 87} (1982), 131--152.



%\bibitem{Bartnik84} R. Bartnik, {\em Existence of maximal surfaces in asymptotically flat spacetimes}, Comm. Math. Phys. {\bf 94} (1984), 155--175.
%
%\bibitem{Bartnik88} R. Bartnik, {\em Regularity of variational maximal surfaces}, Acta Math. {\bf 161} (1988), 145--181.

\bibitem{Bayard03} P. Bayard, {\em Dirichlet problem for space-like hypersurfaces with prescribed scalar curvature in $\mathbb{R}^{n,1}$}, Calc. Var. Partial Differ. Equ. {\bf 18} (2003), 1--30.

\bibitem{CNSI} L. Caffarelli, L. Nirenberg and J. Spruck, {\em The {D}irichlet problem for nonlinear second-order elliptic equations I. Monge-Amp\`{e}re equation.}, Comm. Pure Appl. Math.  {\bf 37} (1984), 369--402.

\bibitem{CNSIII} L. Caffarelli, L. Nirenberg and J. Spruck, {\em The {D}irichlet problem for nonlinear second-order elliptic
	equations {III}., {F}unctions of the eigenvalues of the	{H}essian}, Acta Math.  {\bf 155} (1985), 261--301.

\bibitem{CNSV} L. Caffarelli, L. Nirenberg and J. Spruck, {\em Nonlinear second-order elliptic equations V. The
	Dirichlet problem for Weingarten hypersurfaces}, Comm. Pure Appl. Math. {\bf 41} (1988), 41--70.
	
	
%\bibitem{BJT13} C. Bereanu, P. Jebelean, P.J. Torres, {\em Positive radial solutions for Dirichlet problems with mean curvature operators in Minkowski space}, J. Funct. Anal. {\bf 264} (2013), 270--287.
\bibitem{CDH23}
C. Chen, W. Dong and F. Han, {\em Interior Hessian estimates for a class of Hesian type equations}, Calc. Var.
Partial Differ. Equ. {\bf 62} (2023), 52.

\bibitem{CTX21} L. Chen, Q. Tu and N. Xiang, {\em Pogorelov type estimates for a class of Hessian quotient equations}, J. Differ. Equ. {\bf 282} (2021), 272--284.

\bibitem{CTX20} X. Chen, Q. Tu and N. Xiang, {\em A class of Hessian quotient equations in Euclidean space}, J. Differ. Equ. {\bf 269} (2020), 11172--11194.

\bibitem{CTX23.} X. Chen, Q. Tu and N. Xiang, {\em The Dirichlet problem for a class of Hessian quotient equations on Riemannian manifolds}, Int. Math. Res. Not. (2023),  10013--10036.




%\bibitem{CNSIV} L. A. Caffarelli, L. Nirenberg and J. Spruck, {\em Nonlinear second order elliptic equations. IV. Starshaped compact Weingarten hypersurfaces}, Current topics in partial differential equations, 1--26. Kinokuniya, Tokyo, 1986.



%\bibitem{COOR13} C. Corsato, F. Obersnel, P. Omari, S. Rivetti, {\em Positive solutions of the Dirichlet problem for the prescribed mean curvature equation in Minkowski space}, J. Math. Anal. Appl. {\bf 405} (2013), 227--239.
%
%\bibitem{COOR 13} C. Corsato, F. Obersnel, P. Omari, S. Rivetti, {\em On the lower and upper solution method for the prescribed mean curvature equation in Minkowski space}, DCDS Suppl., Special Issue. {\bf } (2013), 159--168.
%\bibitem{CTX23}
%X. Chen, Q. Tu and N. Xiang, {\em Pogorelov estimates for semi-convex solutions of $k$-curvature equations}, Proc. Amer. Math. Soc., 2023.

\bibitem{CTX25} X. Chen, Q. Tu and N. Xiang, {\em Dirichlet problem for degenerate Hessian quotient type curvature equations}, Calc. Var.
Partial Differ. Equ. {\bf 64} (2025), 99.


%\bibitem{CJ20} J. Chu and H. Jiao, {\em Curvature estimates for a class of Hessian type equations}, Calc. Var. Partial Differential Equations. {\bf 60}, 90 (2021).

\bibitem{Delano90} F. Delano\`e, {\em The {D}irichlet problem for an equation of given {L}orentz-{G}aussian curvature}, Ukrain. Mat. Zh. {\bf 161} (1990), %no. 12, 
1704--1710.translation in Ukrainian Math. J. {\bf 42} (1990), %no. 12
 1538–1545.

%\bibitem{FRT17} D. de la Fuente, A. Romero, P. J. Torres, {\em Existence and extendibility of rotationally symmetric graphs with a prescribed higher mean curvature function in Euclidean and Minkowski spaces}, J. Math. Anal. Appl. {\bf 446} (2017), 1046--1059.

%\bibitem{Gerhardt96}
%C. Gerhardt,
%{\em Closed Weingarten hypersurfaces in Riemannian manifolds},
%J. Differential Geometry  {\bf 43} (1996), 612--641.

%\bibitem{Gerhardt01} C. Gerhardt, {\em  Hypersurfaces of prescribed curvature in Lorentzian manifolds}, Indiana Univ. Math. J. {\bf 49} (2000), 1125-1153.

%\bibitem{Gerhardt03} C. Gerhardt, {\em  Hypersurfaces of prescribed scalar curvature in Lorentzian manifolds}, J. Reine Angew. Math. {\bf 554} (2003), 157-199.

%\bibitem{GG02} B. Guan and P. Guan, {\em Convex hypersurfaces of prescribed curvatures}, Ann. of Math. (2) {\bf 156} (2002), no. 2, 655--673.

\bibitem{Guan98} B. Guan, {\em The Dirichlet problem for Monge-Amp\`{e}re equations in non-convex domains and spacelike hypersurfaces of constant Gauss curvature}, Trans. Amer. Math. Soc. {\bf 350} (1998), 4955--4971.

\bibitem{Guan14} B. Guan, {\em Second order estimates and regularity for fully nonlinear ellitpic equations on Riemannian manifolds}, Duke Math. J. {\bf 163} (2014), 1491--1524.



%\bibitem{GL94}
%P.F. Guan, Y.Y. Li,{\em The Weyl problem with nonnegative Guass curvature}, J. Diffe. Geom., 39(1994), 331-342.

\bibitem{GGQ22} M. George, B. Guan and C. Qiu, , {\em Fully nonlinear elliptic equations on Hermitian manifolds for symmetric functions of partial Laplacians}, J. Geom. Anal. 32 (2022), no. 6, Paper No. 183, 27 pp.

\bibitem{GN23} B. Guan and X. Nie, {\em Fully nonlinear elliptic equations with gradient terms on Hermitian manifolds}, Int. Math. Res. Not. IMRN 2023, no. 16, 14006–14042.





\bibitem{GS04}B. Guan and J. Spruck, {\em Locally convex hypersurfaces of constant curvature with boundary}, Comm. Pure Appl. Math. {\bf 57} (2004), 1311--1331.

%\bibitem{GLL12} P. Guan, J. Li and Y. Li, {\em Hypersurfaces of prescribed curvature measure}, Duke Math. J. {\bf 161} (2012), no. 10, 1927--1942.

\bibitem{GRW15} P. Guan, C. Ren and Z. Wang, {\em Global $C^2$ estimates for convex solutions of curvature equations}, Commun. Pure Appl. Math. {\bf 68} (2015), 1927--1942.

\bibitem{GJ25} M. Guo and H.  Jiao, {\em The Dirichlet problem for a class of curvature equations in Minkowski space}, J. Differ. Equ. {\bf 425} (2025), 129--156.
%\bibitem{GGM21} Y. Gao, Y. L. Gao, J. Mao, {\em The Dirichlet problem for a class of Hessian quotient equations in Lorentz-Minkowski space $\mathbb{R}_1^{n+1}$},  preprint, arXiv: 2111. 02028, 2021.

%\bibitem{HL13} F. R. Harvey and H. B. Lawson, Jr., {\em $p$-convexity, $p$-plurisubharmonicity and the Levi problem},
%    Indiana Univ. Math. J. {\bf 62} (2013), no. 1, 149--169.

%\bibitem{Huang 13} Y. Huang, {\em Curvature estimates of hypersurfaces in the Minkowski space}, Chin. Ann. Math. Ser. {\bf 34} (2013), 753--764.

\bibitem{Ivochkina90} N. Ivochkina, {\em Solution of the Dirichlet problem for equations of $m$-th order curvature. (Russian)}, Mat. Sb. {\bf 180} (1989), %no. 7, 
867--887, %991; 
translation in Math. USSR-Sb. {\bf 67} (1990), %no. 2,
 317--339.

\bibitem{Ivochkina91} N. Ivochkina, {\em The Dirichlet problem for the curvature equation of order $m$}, Algebra i Analiz {\bf 2} (1990), %no. 3, 
192--217; translation in Leningrad Math. J. {\bf 2} (1991), %no. 3, 
631--654.

\bibitem{ILT96}N. Ivochkina, M. Lin and N. Trudinger {\em The Dirichlet problem for the prescribed curvature quotient equations with general boundary values}, Geometric analysis and the calculus of variations, 125--141. International Press, Cambridge, Mass, 1996.

%\bibitem{IT98}
%N.M. Ivochkina, F. Tomi,{\em Locally convex hypersurfaces of prescribed curvature and boundary}, Calc. Var. Partial Differ. Equ. {\bf 7} (1998), 293-314.

%\bibitem{JJ24}
%H. Jiao and Y. Jiao, {\em The Pogorelov estimates for degenerate curvature equations}, Int. Math. Res. Not. IMRN {\bf 18}, (2024), 12504--12529.

\bibitem{JL20} H. Jiao and J. Liu, {\em On a class of Hessian type equations on Riemannian manifolds}, Proc. Amer. Math. Soc. {\bf 151} (2023), %no. 2,
 569--581.

\bibitem{JaoSun22} H. Jiao and Z. Sun, {\em The Dirichlet problem for a class of prescribed curvature equations}, J. Geom. Anal. {\bf 32} (2022), %no. 11, 
261.

\bibitem{JW22} H. Jiao and Z. Wang, {\em The Dirichlet problem for degenerate curvature equations}, J. Funct. Anal. {\bf 283} (2022), %no. 1, 
109485.


%\bibitem{LT94} M. Lin and N. S. Trudinger, {\em On some inequalities for elementary symmetric functions}, Bull. Austral. Math. Soc. {\bf 50} (1994), 317-326.

%\bibitem{MX18} R. Ma, M. Xu, {\em Positive rotationally symmetric solutions for a Dirichlet problem involving the higher mean curvature operator in Minkowski space}, J. Math. Anal. Appl. {\bf 460} (2018), 33--36.
\bibitem{LT94}
M. Lin and N. Trudinger, {\em The Dirichlet problem for the prescribed curvature quotient equations}, Topol. Methods Nonlinear Anal. {\bf 3} (1994), 307--323.

\bibitem{LMZ22} C. Liu, J. Mao  and Y. Zhao, {\em Pogorelov type estimates for a class of Hessian quotient equations in Lorentz-Minkowski space $\mathbb{R}_1^{n+1}$}, J. Differ. Equ. {\bf 327} (2022), 212--225.

\bibitem{LRW17} M. Li, C. Ren and Z. Wang, {\em An interior estimate for convex solutions and a rigidity theorem}, J. Funct. Anal. {\bf 270} (2016), 2691--2714.  

\bibitem{RW19} C. Ren and Z. Wang, {\em On the curvature estimates for Hessian equations}, Amer. J. Math. {\bf 141} (2019), %no. 5, 
1281--1315.

\bibitem{RW20} C. Ren and Z. Wang, {\em The global curvature estimate for the $n-2$ Hessian equation}, Calc. Var. Partial. Differ. Equ. {\bf 62} (2023), %no. 9, 
239.

%\bibitem{RWX24} C. Ren Z. Wang and L. Xiao, {\em The prescribed curvature problem for entire hypersurfaces in Minkowski space}, Anal. PDE {\bf 17} (2024), no. 1, 1-40.

%\bibitem{Sch02} Schn\"{u}rer, {\em The Dirichlet problem for Weingarten hypersurfaces in Lorentz manifolds}, Math. Z. {\bf 242} (2002), 159-181.

%\bibitem{Sha86} J. P. Sha, {\em p-convex Riemannian manifolds}, Invent. Math. {\bf 83} (1986), no. 3, 437--447.

%\bibitem{Sha87} J. P. Sha, {\em Handlebodies and $p$-convexity}, J. Differential Geom. {\bf 25} (1987), no. 3, 353--361.

%\bibitem{SX17} J. Spruck and L. Xiao, {\em A note on starshaped compact hypersurfaces with a prescribed scalar curvature in space forms}, Rev. Mat. Iberoam. {\bf 33} (2017), 547--554.
\bibitem{S05}
J. Spruck,{\em Geometric aspects of the theory of fully nonlinear
elliptic equations}, Clay Mathematics Proceedings, {\bf 2} (2005), 283--309.

%\bibitem{Trudinger95} N. S. Trudinger, {\em On the Dirichlet problem for Hessian equations}, Acta Math. {\bf 175} (1995), 151--164.

%\bibitem{Urbas02} J. Urbas, {\em Hessian Equations on Compact Riemannian Manifolds}, Nonlinear Problems in Mathematical Physics and Related Topics II 367–377. Kluwer/Plenum, New York (2002).

\bibitem{Urbas03} J. Urbas, {\em The Dirichlet problem for the equation of prescribed scalar curvature in Minkowski space}, Calc. Var. Partial Differ. Equ. {\bf 18} (2003), 307--316.

\bibitem{Urbas 03} J. Urbas, {\em Interior curvature bounds for spacelike hypersurfaces of prescribed k-th mean curvature}, Comm. Anal. Geom. {\bf 18} (2003), 307--316.

%\bibitem{WX22} Z. Wang and L. Xiao, {\em Entire spacelike hypersurfaces with constant $\sigma_{k}$ curvature in Minkowski space},  Math. Ann. {\bf 382} (2022), 1279--1322.

%\bibitem{Wu87} H. Wu, {\em Manifolds of partially positive curvature}, Indiana Univ. Math. J. {\bf 36} (1987), no. 3, 525--548.

\bibitem{WangB25} B. Wang, {\em The Dirichlet problem for the prescribed curvature equations in Minkowski space}, J. Funct. Anal. {\bf 289} (2025), 111149. 


\bibitem{WangZZ} Z. Wang, {\em The global curvature estimates for Hessian equations}, Proceedings of the International Consortium of Chinese Mathematicians 2018 (2020), %pp.
 367--384.
 
%\bibitem{Yuan21} R. Yuan, {\em On the partial uniform ellipticity and complete conformal metrics with prescribed
% 	curvature functions on manifolds with boundary}, arXiv:2011.08580. %pp.
 	
%\bibitem{Yuan22} R. Yuan, {\em The partial uniform ellipticity and prescribed problems on the conformal classes of
%	complete metrics}, arXiv:2203.13212.

%\bibitem{Yuan23} R. Yuan, {\em An extension of prescribed problems on the conformal classes of complete metrics},
%	arXiv:2304.12835.
	
\bibitem{Yuan25a} R. Yuan, {\em On a level set version of partial uniform ellipticity and applications to a Krylov type equation}, Pure Appl. Math. Q. 21 (2025), no. 6, 2229–2256.
	
\bibitem{Yuan25b} R. Yuan, {\em The structure of fully nonlinear equations and its applications to prescribed problems
	on complete conformal metrics}, arXiv:2503.18757.

\bibitem{Zhou24} J. Zhou, {\em Curvature estimates for a class of Hessian quotient type curvature equations}, Calc. Var. Partial Differ. Equ. {\bf 63} (2024), 88.

%\bibitem{Wang20} Z. Wang, {\em The global curvature estimates for Hessian equations}, Int. Press, Boston, MA, {\bf } (2020), 367--384.
 



\end{thebibliography}
\end{document}